\documentclass{amsart}
\usepackage{amssymb,latexsym,amsfonts}
\usepackage{amsthm}
\usepackage{fontenc}
\usepackage{amsmath}
\usepackage{textcomp}
\usepackage{graphicx}
\usepackage{stmaryrd}
\usepackage{geometry}
\usepackage[all]{xy}
\usepackage[francais]{babel}
\usepackage[latin1]{inputenc}
\usepackage[T1]{fontenc}
\DeclareMathOperator{\Jac}{Jac}
\DeclareMathOperator{\Fitt}{Fitt}

\DeclareMathOperator{\Div}{Div}

\DeclareMathOperator{\Kern}{Ker}
\DeclareMathOperator{\Bun}{Bun_{G}}
\DeclareMathOperator{\Bungm}{Bun_{\mathbb{G}_{m}^{l}}}

\DeclareMathOperator{\Bunb}{Bun_{\textbf{G}}}
\DeclareMathOperator{\Spec}{Spec}
\DeclareMathOperator{\inv}{inv}
\DeclareMathOperator{\supp}{supp}
\DeclareMathOperator{\Id}{Id}
\DeclareMathOperator{\Hom}{Hom}

\DeclareMathOperator{\End}{End}
\DeclareMathOperator{\dime}{dim}
\DeclareMathOperator{\Ima}{Im}

\DeclareMathOperator{\NN}{\mathbb{NN}}

\DeclareMathOperator{\Gr}{Gr}

\DeclareMathOperator{\Aut}{Aut}
\DeclareMathOperator{\ad}{ad}
\DeclareMathOperator{\Ad}{Ad}

\DeclareMathOperator{\Ann}{Ann}

\DeclareMathOperator{\coker}{Coker}
\DeclareMathOperator{\Ker}{Ker}

\DeclareMathOperator{\Tr}{Tr}

\DeclareMathOperator{\rg}{rg}

\DeclareMathOperator{\out}{Out}
\DeclareMathOperator{\Ext}{Ext}
\DeclareMathOperator{\Res}{Res}
\DeclareMathOperator{\Gal}{Gal}

\DeclareMathOperator{\val}{val}

\DeclareMathOperator{\Lie}{Lie}

\DeclareMathOperator{\Fr}{Fr}

\newtheorem{thm}{Théorème}
\newtheorem{prop}[thm]{Proposition}
\newtheorem{lem}[thm]{Lemme}
\newtheorem{defi}[thm]{Définition}

\newtheorem{cor}[thm]{Corollaire}

\newcommand{\rmq}{\noindent\textbf{Remarque :}}
\newcommand{\rmqs}{\noindent\textbf{Remarques :}}
\newcommand{\cH}{\mathcal{H}}

\newcommand{\GL}{GL}

\newcommand{\co}{\mathcal{O}}
\newcommand{\cP}{\mathcal{P}}
\newcommand{\cD}{\mathcal{D}}
\newcommand{\cT}{\mathcal{T}}

\newcommand{\cQ}{\mathcal{Q}}
\newcommand{\cm}{\mathcal{M}}

\newcommand{\g}{\gamma}

\newcommand{\la}{\lambda}

\newcommand{\eps}{\epsilon}

\newcommand{\cmdbD}{\overline{\mathcal{M}}_{\lambda}^{\flat,\leq d}}

\newcommand{\cmdo}{\mathcal{M}_{\lambda}}

\newcommand{\cmd}{\overline{\mathcal{M}}_{\lambda}}
\newcommand{\cmdb}{\overline{\mathcal{M}}_{\lambda}^{\flat}}

\newcommand{\kc}{\mathfrak{C}_{+}}

\newcommand{\kC}{\mathfrak{C}}

\newcommand{\kcd}{\mathfrak{C}_{+}^{\lambda}}
\newcommand{\kg}{\mathfrak{g}}
\newcommand{\kq}{\mathfrak{q}}
\newcommand{\kx}{\mathfrak{X}}
\newcommand{\km}{\mathfrak{m}}

\newcommand{\kt}{\mathfrak{t}}

\newcommand{\ev}{ev}
\newcommand{\bg}{\textbf{G}}
\newcommand{\bz}{\textbf{Z}}
\newcommand{\bC}{\textbf{C}}
\newcommand{\bB}{\textbf{B}}

\newcommand{\bh}{\textbf{H}}
\newcommand{\bt}{\textbf{T}}
\newcommand{\bx}{\textbf{x}}
\newcommand{\bw}{\textbf{W}}
\newcommand{\bU}{\textbf{U}}

\newcommand{\abd}{\mathcal{A}_{\lambda}}
\newcommand{\abdD}{\mathcal{A}_{\lambda}^{\leq d}}

\newcommand{\abdb}{\mathcal{A}_{\lambda}^{\flat}}
\newcommand{\abdbD}{\mathcal{A}_{\lambda}^{\flat,\leq d}}

\newcommand{\abdh}{\mathcal{A}_{\lambda}^{\heartsuit}}

\newcommand{\abdd}{\mathcal{A}_{\lambda}^{\diamondsuit}}

\newcommand{\fx}{F_{x}}

\newcommand{\ab}{\mathbb{A}}

\begin{document}
\title{La fibration de Hitchin-Frenkel-Ngô et son complexe d'intersection}
\author{Alexis Bouthier}
\maketitle
\tableofcontents
\selectlanguage{francais}
\begin{center}
\textbf{Abstract:}
\end{center}
In this article, we construct the Hitchin fibration for groups following the scheme outlined by Frenkel-Ngô \cite{FN} in the case of $SL_{2}$. This construction uses as a decisive tool the Vinberg's semigroup and follows the study accomplished in \cite{Bt}. The total space of Hitchin is obtained by taking the fiber product of the Hecke stack with the diagonal of the stack of $G$-bundles $\Bun$; we prove a transversality statement between the intersection complex of the Hecke stack and the diagonal of $\Bun$, over a sufficiently big open subset, in order to get local applications, such that the fundamental lemma for the spherical Hecke algebra.
Along the proof of this theorem, we establish a result concerning the integral conjugacy classes of the points of a simply connected group in a local field.
\bigskip
\bigskip
\begin{center}
\textbf{Résumé:}
\end{center}
Dans cet article, on construit la fibration de Hitchin pour les groupes d'après celle esquissée par Frenkel-Ngô \cite{FN} dans le cas de $SL_{2}$. Cette construction utilise de manière cruciale le semi-groupe de Vinberg et fait suite à l'étude menée dans \cite{Bt}. L'espace total de Hitchin s'obtient comme le produit fibré du champ de Hecke avec la diagonale  du champ des $G$-torseurs $\Bun$; nous démontrons alors un énoncé de transversalité du complexe d'intersection du champ de Hecke avec cette diagonale, au-dessus d'un ouvert suffisamment gros, pour obtenir des applications locales, telles que le lemme fondamental pour l'algèbre de Hecke sphérique.
Dans le cours de la preuve de ce théorème, nous établissons également un énoncé sur les classes de conjugaisons entières des points d'un groupe simplement connexe sur un corps local.
\bigskip
\section*{Introduction}
Dans sa preuve du lemme fondamental pour les algèbres de Lie \cite{N}, Ngô utilise la fibration de Hitchin comme un moule géométrique pour les intégrales orbitales de la fonction caractéristique du compact maximal.
Dans ce travail, on s'intéresse à la construction d'un analogue de la fibration de Hitchin pour le cas des groupes ainsi qu'au lien entre cet espace et les intégrales orbitales d'une autre classe de fonctions, celles de l'algèbre de Hecke sphérique.
\medskip

Soit $k$ un corps algébriquement clos ou un corps fini, soient $\co:=k[[\pi]]$ et $F=k((\pi))$. On note $D:=\Spec(\co)$ et $D^{\bullet}:=\Spec(F)$.

Pour alléger l'introduction, on considère un groupe $G$ semisimple simplement connexe déployé, $(T,B)$ une paire de Borel, $r=\rg (T)$, $W$ le groupe de Weyl et $w_{0}$ son élément long.
Soit $\Delta=\{\alpha_{1},\dots,\alpha_{r}\}$ l'ensemble des racines simples, ainsi que le cône des cocaractères (resp. caractères) dominants $X_{*}(T)^{+}\subset X_{*}(T)$, (resp. $X^{*}(T)^{+}\subset X^{*}(T)$). Enfin, on note $\omega_{1},\dots\omega_{r}$, les poids fondamentaux.
On pose $K:=G(\co)$.
Considérons la grassmannienne affine, 
\begin{center}
$\Gr:=G(F)/K$,
\end{center}
que l'on peut munir d'une structure d'ind-schéma d'après \cite[Prop.2]{H}. On a une interprétation modulaire de la grassmanienne affine $\Gr$, elle classifie les paires $(E,\beta)$ où $E$ est un $G$-torseur sur $D$ et $\beta$ une trivialisation sur $D^{\bullet}$.
De plus, elle admet une décomposition en $K$-orbites, dite de Cartan:
\begin{center}
$\Gr:=\coprod\limits_{\la\in X_{*}(T)^{+}} K\pi^{\la}K/K$.
\end{center}
On introduit alors $\Gr_{\la}:=K\pi^{\la}K/K$ ainsi que $\overline{\Gr}_{\la}$ l'adhérence $\Gr_{\la}$ dans $\Gr$.
Nous avons la description suivante:
\begin{center}
$\overline{\Gr}_{\la}=\coprod\limits_{\mu\leq\la}\Gr_{\mu}$.
\end{center}
En particulier, pour deux $G$-torseurs $E$, $E'$ sur le disque  et $\beta$ un isomorphisme sur $D^{\bullet}$ entre $E$ et $E'$, on obtient  un point:
\begin{center}
$\inv(E,E',\beta)\in K\backslash \Gr=X_{*}(T)^{+}$.
\end{center}
\medskip

Supposons maintenant que $k$ est fini. Soit un entier premier $l\neq p$, on considère l'algèbre de Hecke sphérique
\begin{center}
$\mathcal{H}:=C_{c}(K\backslash G(F)/K,\overline{\mathbb{Q}}_{l})$,
\end{center}
avec comme produit, la convolution des fonctions. Par le dictionnaire fonctions-faisceaux, elle admet une  base, due à Lusztig, donnée par les fonctions $\phi_{\la}$ qui correspondent aux faisceaux $IC_{\overline{\Gr}_{\la}}$.
\medskip
Nous cherchons donc à construire une fibration, analogue à la fibration de Hitchin, dont le nombre de points nous fournirait les intégrales orbitales des fonctions $\phi_{\la}$. Elle s'obtient de la manière suivante.
\medskip

Soit $X$ une courbe projective lisse géométriquement connexe sur $k$, on note $F$ son corps de fonctions, $\mathbb{A}$ l'anneau des adèles et $\co_{\mathbb{A}}$ les adèles entières.
On considère une somme formelle $\la=\sum\limits_{x\in X}\la_{x}[x]$ avec $\la_{x}\in X_{*}(T)^{+}$ presque tous nuls et $S=\supp(\la):=\{x\in X\vert~ \la_{x}\neq 0\}$.
Enfin, on pose $\overline{\Gr}_{\la}:=\prod\limits_{x\in X}\overline{\Gr}_{\la_{x}}$.

A la suite de Beilinson-Drinfeld \cite{BD}, on considère le champ de Hecke $\overline{\cH}_{\la}$ qui classifie les triplets $(E,E',\beta)$ où $E$, $E'$ sont des $G$-torseurs sur $X$ et un isomorphisme 
\begin{center}
$\beta: E_{\vert X-S}\rightarrow E'_{\vert X-S}$,
\end{center}
tel que pour tout $x\in S$, $\inv_{x}(E_{\vert D_{x}},E'_{\vert D_{x}},\beta_{D_{x}^{\bullet}})\leq \la_{x}$.
On forme alors le carré cartésien:
$$\xymatrix{\cmd\ar[d]\ar[r]^{\Delta}&\overline{\cH}_{\la}\ar[d]\\\Bun\ar[r]^-{\Delta}&\Bun\times\Bun}$$
On a la description adélique suivante: 
\begin{center}
$\cmd(k):=G(F)\backslash\{(\g,(g_{x}))\in G(F)\times G(\ab)/G(\co_{\ab})\vert~ g_{x}^{-1}\g g_{x}\in\overline{K_{x}\pi_{x}^{\la}K_{x}}\}$,
\end{center}
où $G(F)$ agit par $h.(\g,(g_{x}))=((h\g h^{-1},(hg_{x})))$.
Si $k$ est fini, si l'on considère le complexe $\Delta^{*}IC_{\overline{\cH}_{\la}}$, la trace de Frobenius en un point $t=(\g,(g_{x}))$ est donnée par:
\begin{center}
$\Tr(\Fr_{t},\Delta^{*}IC_{\overline{\cH}_{\la}})=\bigotimes\limits_{x\in X}' \phi_{\la_{x}}(g_{x}^{-1}\g g_{x})$.
\end{center}
Si l'on regarde au travers des représentations fondamentales $\rho_{i}: G\rightarrow\GL(V_{\omega_{i}})$, la condition $\inv_{x}(E,E,\phi)\leq \la_{x}$ se récrit:
\begin{center}
$\forall~ i, \rho_{i}(\phi)\in H^{0}(X,\End(\rho_{i}(E))(\left\langle \omega_{i},-w_{0}\la\right\rangle))$,
\end{center}
où $\rho_{i}(E)$ est le fibré vectoriel que l'on obtient en poussant $E$ par $\rho_{i}: G\rightarrow\GL(V_{\omega_{i}})$.
On a alors une application introduite par Frenkel-Ngô :
\begin{center}
$f:\cmd\rightarrow\abd:=\bigoplus\limits_{i=1}^{r}H^{0}(X,\co_{X}(\left\langle \omega_{i},-w_{0}\la\right\rangle))$
\end{center}
donnée par $(E,\phi)\mapsto (\Tr(\rho_{i}(\phi)))_{1\leq i\leq r}$,
réminiscente de la fibration de Hitchin.
\bigskip
Si l'on suppose que $k$ est fini et  $f$ propre, en calculant la trace de Frobenius en un point $a\in\abd$  de $f_{*}\Delta^{*}IC_{\overline{\cH}_{\la}}$, on obtient  une intégrale orbitale globale stable :
\begin{center}
$\Tr(\Fr_{a},f_{*}\Delta^{*}IC_{\overline{\cH}_{\la}})=SO_{a}(\phi_{\la})$.
\end{center}
Pour pouvoir utiliser des techniques similaires à celles de Ngô \cite{N}, on aimerait savoir si $\Delta^{*}IC_{\overline{\cH}_{\la}}$ est pervers et pur, comme cela a été formulé par Frenkel et Ngô dans la conjecture 4.1 de \cite{FN}.
En fait, nous allons démontrer que l'on obtient le complexe d'intersection sur $\cmd$, quitte à considérer un certain ouvert.
Ceci fait l'objet de notre premier théorème.

On définit un ouvert $\abdb\subset\abd$; il consiste à borner la taille du discriminant local. Soit $\cmdb$ l'ouvert de $\cmd$ correspondant. On note toujours $\Delta:\cmdb\rightarrow\overline{\cH}_{\la}$ la flèche naturelle. Notre théorème principal est l'énoncé de transversalité du complexe d'intersection de $\overline{\cH}_{\la}$ avec la diagonale $\Delta$ au-dessus de l'ouvert $\abdb$ :
\begin{thm}\label{2}
Soit $G$ un schéma en groupes quasi-déployé sur $X$, tel que $G_{der}$ est simplement connexe sans facteur simple de type $A_{2r}$, alors le champ $\cmdb$ est  équidimensionnel, on note $d$ sa codimension dans $\overline{\cH}_{\la}$ et nous avons l'égalité suivante :
\begin{center}
$\Delta^{*}[-d]IC_{\overline{\cH}_{\la}}=IC_{\cmdb}$.
\end{center}
\end{thm}
Cet ouvert $\abdb$ est suffisamment gros pour pouvoir établir des identités locales telles que le lemme fondamental. Si l'on part d'une situation locale avec un $\la$ et un discriminant fixé, on peut globaliser cette situation en ajoutant de grands \og cocaractères \fg avec des discriminants transverses en des points auxiliaires pour pouvoir tomber dans l'ouvert $\abdb$.

A la suite de la preuve du théorème, nous démontrons un résultat local qui peut se révéler utile pour d'autres applications que celles qui occupent notre propos.$\\$
On pose pour la fin de ce paragraphe $\co=k[[\pi]]$ et $F=k((\pi))$. Il s'agit de décrire les classes de $G(\co)$-conjugaison à l'intérieur de $G(F)^{rs}$. 
Elles font naturellement intervenir le semigroupe de Vinberg, introduit par Vinberg en caractéristique nulle \cite{Vi} et par Rittatore en caractéristique $p$ \cite{Ri}.

C'est un schéma affine, normal, intègre qui contient $G_{+}:=(T\times G)/Z_{G}$ comme ouvert dense, où l'on plonge le centre $Z_{G}$ de $G$ antidiagonalement.
Expliquons désormais le lien entre les strates de Cartan et le semi-groupe de Vinberg.
Soit $\la\in X_{*}(T)^{+}$, on définit alors un sous-schéma localement fermé $V_{G}^{\la,0}$ de $V_{G}$ sur $\Spec(\co)$ qui consiste en des paires $(\pi^{-w_{0}\la},x)$.
Nous avons  que $x\in K\pi^{\la}K$ si et seulement si $x_{+}:=(\pi^{-w_{0}\la},x)\in V_{G}^{\la,0}(\co)$.
D'après Steinberg \cite{S}, on dispose d'un morphisme, dit de \og polynôme caractéristique \fg:
\begin{center}
$\chi: G\rightarrow T/W$,
\end{center}
le résultat est alors le suivant:
\begin{thm}\label{4bis}
Soient $\g,\g'\in K\pi^{\la} K\cap G(F)^{rs}$ telles que $\chi(\g)=\chi(\g')$.

Soit $d_{+}:=\left\langle 2\rho,\la\right\rangle+d(\g)$ où $d(\g)$ est la valuation du discriminant de $\g$. Alors $d_{+}$ est un entier positif.
De plus, si l'on considère $\g_{+}:=(\pi^{-w_{0}\la},\g) \in V_{G}^{\la,0}(\co)$ (resp. $\g_{+}'$), alors les assertions suivantes sont équivalentes:
\begin{enumerate}
\item
$\g$ et $\g'$ sont conjuguées sous $G(\co)$.
\item
$\bar{\g}_{+}$ et $\bar{\g}_{+}'$ sont conjuguées modulo $\pi^{2d_{+}+1}$,
où $\bar{\g}_{+}$ (resp. $\bar{\g}'_{+}$) désigne les réductions à $V_{G}^{\la,0}(\co/\pi^{2d_{+}+1}\co)$.
\end{enumerate}
\end{thm}

Passons en revue l'organisation de l'article.
On commence par faire des rappels issus de \cite{Bt} sur le semigroupe de Vinberg $V_{G}$ ainsi que sur le quotient adjoint $\chi_{+}$. Puis, nous définissons la fibration de Hitchin pour les groupes.
Une première définition immédiate de l'espace total de Hitchin s'obtient en prenant le produit fibré le long de la diagonale de $\Bun$ du champ de Hecke $\overline{\cH}_{\la}$.  Nous réinterprétons ensuite cette définition en fonction du semi-groupe de Vinberg pour en déduire une définition de la fibration de Hitchin $f_{\la}:\cmd\rightarrow\abd$ sur laquelle agit un champ de Picard $\cP\rightarrow\abd$.
On s'intéresse alors à l'étude du complexe d'intersection du champ $\cmd$. Dans un premier temps, nous introduisons l'ouvert transversal $\abdd\subset\abd$ où la valuation du discriminant est partout inférieure ou égale à un. Pour un tel ouvert, la situation est très simple puisque la fibration $f_{\la}$ est lisse.
Nous considérons donc un ouvert plus gros $\abdb\subset\abd$ qui contient des polynômes caractéristiques dont la valuation du discriminant local peut-être plus grande que un. Nous démontrons alors le théorème \ref{2}  de transversalité sur cet ouvert. Pour obtenir un tel théorème, on commence par construire un modèle local des singularités de l'espace de Hitchin. Nous utilisons ensuite un théorème d'approximation des polynômes caractéristiques ainsi que des énoncés de relèvements d'Elkik et Gabber-Ramero pour comparer ce modèle local avec la strate $\overline{\Gr}_{\la}$.
C'est en étudiant ce résultat d'approximation des polynômes caractéristiques qu'est apparu naturellement le résultat \ref{4bis}.
Enfin, en appendice, nous donnons la preuve d'un théorème de Gabber non publié sur les problèmes de relèvements pour des morphismes finis et plats.
\medskip

Je remercie très chaleureusement Gérard Laumon et Bao Châu Ngô pour les nombreuses discussions que nous avons pu avoir sur cet article et particulièrement pour leur soutien dans les derniers mois pour m'aider à obtenir le résultat final. J'exprime également ma profonde gratitude à Ofer Gabber pour  son concours indispensable concernant les problèmes de relèvement, plus particulièrement l'appendice et l'énoncé \ref{inftem}. Je remercie Y. Varshavsky de m'avoir suggéré l'amélioration du théorème principal. Enfin, je remercie l'Université de Chicago pour les nombreux séjours que j'ai pu y effectuer.

\section{Rappels sur le semi-groupe de Vinberg}\label{introsemi}
Soit $k$ un corps.
On considère un groupe connexe réductif $\bg$  tel que $\bg_{der}$ soit simplement connexe, déployé sur $k$. On note $G$ une forme quasi-déployée de $\bg$ sur un $k$-schéma $X$. Soit $(\bB, \bt)$ une paire de Borel de $\bg$, $\Delta$ l'ensemble des racines simples, $r$ le rang semisimple de $\bg$ et $\bw$ le groupe de Weyl.
On commence par construire le semi-groupe de Vinberg $V_{\bg}$ sur $k$.
Choisissons une base $\omega_{1}',\dots,\omega_{l}'$ du réseau des caractères:
\begin{center}
$\chi:\bt\rightarrow\mathbb{G}_{m}$
\end{center}
tels que
\begin{center}
$\forall~\alpha\in\Delta=\{\alpha_{1},\dots,\alpha_{l}\}, \left\langle \chi,\check{\alpha}\right\rangle=0.$
\end{center}
Chaque $\omega_{i}':\bt\rightarrow\mathbb{G}_{m}$, $1\leq i\leq l$, se prolonge de manière unique en un caractère:
\begin{center}
$\omega_{i}':\bg\rightarrow\mathbb{G}_{m}$.
\end{center}
Pour tout indice $i, 1\leq i\leq l$, notons $V_{\omega_{i}'}$ un espace vectoriel de dimension un sur lequel $\bg$ agit par $\omega_{i}':\bg\rightarrow\mathbb{G}_{m}$.
L'ensemble des poids dominants est stable par translation par les éléments du réseau $\mathbb{Z}\omega_{1}'+\dots+\mathbb{Z}\omega_{l}'$. Le quotient par ce réseau est un cône saturé non dégénéré de $X^{*}(\bt)/(\mathbb{Z}\omega_{1}'+\dots+\mathbb{Z}\omega_{l}')$ qui est engendré par $\bar{\omega}_{1},\dots,\bar{\omega}_{r}$ que l'on relève en une famille de caractères $\omega_{1},\dots,\omega_{r}$ de $\bt$.

Soit $(\rho_{\omega_{i}}, V_{\omega_{i}})$ la représentation irréductible de plus haut poids $\omega_{i}$.
Considérons $\bg_{+}:=(\bt\times \bg)/Z_{\bg}$ où $Z_{\bg}$ se plonge par $\la\rightarrow (\la,\la^{-1})$ et $\bt_{+}=(\bt\times \bt)/Z_{\bg}$.
On a une immersion :
$$\begin{array}{ccccc}
&  & \bg_{+} & \to & \prod\limits_{i=1}^{r}\End(V_{\omega_{i}})\times\prod\limits_{i=1}^{l}\Aut(V_{\omega_{i}'})\times\prod\limits_{i=1}^{r}\mathbb{A}_{\alpha_{i}} \\
& & (t,g) & \mapsto & (\omega_{i}(t)\rho_{\omega_{i}}(g),\omega_{i}'(tg),\alpha_{i}(t)).\\
\end{array}$$
Ici,  $H_{\bg}:=\prod\limits_{i=1}^{r}\End(V_{\omega_{i}})\times\prod\limits_{i=1}^{l}\Aut(V_{\omega_{i}'})\times\prod\limits_{i=1}^{r}\mathbb{A}^{1}_{\alpha_{i}}$
et $H_{\bg}^{0}$ sera la même chose où l'on enlève $\{0\}$ dans chaque $\End(V_{\omega_{i}})$.
On pose alors $V_{\bg}$ (resp $V_{\bg}^{0}$) la normalisation de l'adhérence de $\bg_{+}$ dans $H_{\bg}$ (resp. $H_{\bg}^{0})$ et $V_{\bt}$ l'adhérence de $\bt_{+}$ dans $V_{\bg}$, lequel est normal par \cite[Cor. 6.2.14]{BK}.
On a un théorème analogue à celui de Chevalley \cite[Prop. 1.3]{Bt}:
\begin{thm}\label{bouth}
L'application de restriction $\phi:k[V_{\bg}]^{\bg}\rightarrow k[V_{\bt}]^{\bw}$ est un isomorphisme de $k$-algèbres. De plus, $V_{\bt}/\bw$ est un espace affine de dimension $2r$ avec un tore, dont les coordonnées sont données par les $\omega_{i}'$, $(\alpha_{i},0)$, $\chi_{i}=\Tr(\rho_{(\omega_{i}, \omega_{i})})$.
\end{thm}
On en déduit alors un morphisme 
\begin{center}
$\chi_{+}: V_{\bg}\rightarrow\bC_{+}:=V_{\bt}/W$.
\end{center}
Nous avons également une flèche $\chi: \bg\rightarrow \bt/\bw=\mathbb{G}_{m}^{l}\times\mathbb{A}^{r}$, issue du théorème de Chevalley \cite[Th.6.1]{S}.
Steinberg a construit une section à cette flèche de la manière suivante; on commence par supposer que $\bg$ est semisimple simplement connexe, dans ce cas $\bt/\bw=\mathbb{A}^{r}$.
Pour un $r$-uplet $(a_{1},.., a_{r})\in \bt/\bw:=\mathbb{A}^{r}$, on définit:
\begin{center}
$\epsilon(a_{1},.., a_{r}):=\prod\limits_{i=1}^{r}x_{\alpha_{i}}(a_{i})n_{i}$,
\end{center}
où les $x_{\alpha_{i}}(a_{i})$ sont des éléments du groupe radiciel $U_{\alpha_{i}}$ et les $n_{i}$ sont des éléments du normalisateur $N_{\bg}(\bt)$ représentant les réflexions simples $s_{\alpha_{i}}$ de $\bw$.\\
Ainsi, $\epsilon(a)\in\prod\limits_{i=1}^{r}\bU_{i}n_{i}$ et en utilisant les relations de commutation on a que :
\begin{center}
$\prod\limits_{i=1}^{r}\bU_{i}n_{i}=\bU_{w}w$
\end{center}
où $w=s_{1}s_{2}...s_{r}$ et $\bU_{w}=\bU\cap w\bU^{-}w^{-1}$.
Voyons comment on l'étend au cas réductif.

Soit $S$ un sous-tore de $\bt$ de telle sorte que $\bt=S\times \bt_{der}$, alors $\bg=S\ltimes \bg_{der}$.
La flèche de Steinberg est donnée par:
\begin{center}
$\chi:\bg\rightarrow S\times\mathbb{A}^{r}$,
\end{center}
où $\mathbb{A}^{r}$ est la partie correspondant au quotient adjoint de $\bg_{der}$.
Soit $\bg^{reg}:=\{(g,\g)\in \bg\times\bg\vert~ g\g g^{-1}=\g\}$.
Nous avons alors le théorème suivant \cite[Th. 8.1]{S} et \cite[Prop. 2.5]{dC-M}:
\begin{prop}
Soit $\epsilon_{\bg_{der}}$ la section de Steinberg pour $\bg_{der}$.
On pose alors $\eps_{\bg}:S\times\mathbb{A}^{r}\rightarrow$ donnée par
\begin{center}
$\eps_{\bg}(s,a)=s\epsilon_{\bg_{der}}(a)$.
\end{center}
Alors $\eps_{\bg}$ est une section à $\chi$ et tombe dans $\bg^{reg}$.
\end{prop}
On considère maintenant le schéma en groupes des centralisateurs:
\begin{center}
$I:=\{(g,\g)\in \bg\times V_{\bg}\vert~ g^{-1}\g g=\g\}$
\end{center}
Voyons  comment on construit une section pour le semigroupe de Vinberg $V_{\bg}$.
Soit la flèche $\phi:\bt\rightarrow \bt_{+}$, $t\rightarrow (t,t^{-1})$. L'image $\bt_{\Delta}$, le tore antidiagonal est isomorphe à $\bt^{ad}:=\bt/Z_{\bg}$, et on a un isomorphisme canonique donné par:
\begin{center}
$\alpha_{\bullet}:\bt_{\Delta}\rightarrow\mathbb{G}_{m}^{r}$.
\end{center}
Soit alors $\psi$ l'isomorphisme inverse, nous opterons pour la notation indiciaire. Ainsi pour $b\in \mathbb{G}_{m}^{r}$, on a: 
\begin{center}
$\forall~ 1\leq i\leq r, \alpha_{i}(\psi_{b})=b_{i}$ et $\psi_{b}\in \bt_{\Delta}$.
\end{center}
On définit $\epsilon_{+}:\mathbb{G}_{m}^{r}\times (\mathbb{G}_{m}^{l}\times\mathbb{A}^{r})\rightarrow \bg_{+}$ par: 
\begin{center}
$\epsilon_{+}(b,a)=\epsilon_{\bg}(a)\psi_{b}$.
\end{center}
Dans la suite, on pose $\cQ_{+}:=\cQ\bt_{\Delta}$ où $\cQ:=\eps_{\bg}(\bt/\bw)$.
On obtient alors le théorème de structure suivant tiré de \cite[Thm.2-3]{Bt}:
\begin{thm}\label{bouth2}
Soit $V_{\bg}^{reg}:=\{\g\in V_{\bg}\vert~\dim I_{g}=r\}$.
La section  $\eps_{+}$ se prolonge en un morphisme $\eps_{+}:\bC_{+}\rightarrow V_{\bg}^{reg}$.

De plus, le morphisme $\chi_{+}^{reg}$ est lisse et ses fibres géométriques sont des $G$-orbites.
Enfin, il existe un unique schéma en groupes commutatifs $J$, lisse sur $\bC_{+}$ et muni d'un morphisme $\chi_{+}^{*}J\rightarrow I$, qui est un isomorphisme sur $V_{\bg}^{reg}$.
\end{thm}

$\rmq$
\begin{enumerate}
\item
Il est à noter que la section $\eps_{+}$ prolongée au semigroupe de Vinberg dépend fortement de l'ordre des facteurs $U_{i}n_{i}$. En effet, un ordre différent amène à des sections qui ne sont plus nécessairement conjuguées, comme on peut le voir au point au point $(0,0)\in\bC_{+}$.
\item
Dans \cite{Bt}, la preuve est faite dans le cas semisimple mais s'étend telle quelle au cas réductif.
\end{enumerate}

On dispose également d'un morphisme, dit d'abélianisation:
\begin{center}
$\alpha: V_{\bg}\rightarrow A_{\bg}:=\mathbb{A}^{r}$
\end{center}
donnée par $(t,g)\mapsto(\alpha_{i}(t))_{1\leq i\leq r}$.

Nous avons maintenant besoin de passer au cas quasi-déployé.
On considère un épinglage $(\bB,\bt,\textbf{x}_{+})$ de $\bg$ avec $\bx_{+}=\sum\limits_{\alpha\in\Delta}\bx_{\alpha}$, où $\textbf{x}_{\alpha_{i}}$ est un vecteur propre du sous-espace radiciel $\Lie (\bU)_{\alpha}$ de $\Lie(\bU)$ correspondant à la racine $\alpha$ de $\bt$ et où $\bU$ est le radical unipotent de $\bB$.
Notons $\out(\bg)$ le groupe des automorphismes laissant invariant cet épinglage. C'est un groupe discret, éventuellement infini. Il agit sur l'ensemble des racines, en laissant l'ensemble $\Delta$ des racines simples. Il agit également sur $\bw$ de façon compatible avec l'action de $\bw$ sur $\bt$.

La donnée de la forme quasi-déployée $G$ de $\bg$ revient à la donnée d'une flèche 
\begin{center}
$\rho:\pi_{1}(X,x)\rightarrow\out(\bg)$.
\end{center}
où $x$ est un point géométrique du $X$.
On définit alors le semi-groupe de Vinberg de $G$ comme le $X$-schéma:
\begin{center}
$V_{G}:= \rho\wedge^{\out(\bg)}V_{\bg}$
\end{center}
ainsi que l'abélianisé:
\begin{center}
$A_{G}:=\rho\wedge^{\out(\bg)}A_{\bg}$.
\end{center}
L'action de $\bw\rtimes\out(\bg)$ sur $\bt$ induit une action de $\out(\bg)$ sur $\bC_{+}$ et on définit l'espace des polynômes caractéristiques par:
\begin{center}
$\kc:=\rho\wedge^{\out(\bg)}\bC_{+}$.
\end{center}
On  a des morphismes:
\begin{center}
$\chi_{+}:V_{G}\rightarrow\kc$.
\end{center}
et 
\begin{center}
$\alpha:V_{G}\rightarrow A_{G}$
\end{center}
Enfin, la flèche finie plate $\theta_{\bg}:V_{\bt}\rightarrow\bC_{+}$ qui est génériquement étale galoisienne de groupe $\bw$ induit une flèche
\begin{center}
$\theta:V_{T}\rightarrow\kc$
\end{center}
qui est génériquement un torseur sous le schéma en groupes fini étale $W=\rho\wedge^{\out(\bg)}\bw$.

Il nous faut maintenant obtenir une section. L'inconvénient de la section de Steinberg $\eps$ est qu'elle n'est pas $\out(\bg)$-équivariante, néanmoins $\eps_{+}$ l'est comme on va le voir.
Il ne  nous reste plus qu'à examiner le comportement de la section par torsion extérieure. Si nous regardions seulement la section de Steinberg pour une forme quasi-déployée de $\bg$, nous avons le résultat suivant dû à Steinberg.

\begin{thm}\cite[Th. 9.4]{S}
Soit $G$ un groupe connexe réductif avec $G_{der}$ simplement connexe, quasi-déployé sur un $k$-schéma $X$ qui ne contient pas de composante simple de type $A_{r}$, $r$ pair, alors la section de Steinberg $\cQ=\eps(T/W)$ est définie sur $X$.
\end{thm}

A la suite de Donagi-Gaitsgory \cite{DG} et Ngô \cite[sect. 2.4]{N}, nous donnons une interprétation alternative du centralisateur régulier.
On suppose de plus la caractéristique du corps $k$ est première à l'ordre de $\bw$.
On a un morphisme fini plat $W$-équivariant:
\begin{center}
$\theta:V_{T}\rightarrow\kc$
\end{center}
ramifié le long du diviseur $\mathfrak{D}=\bigcup\limits_{\alpha\in R}\overline{\Kern(\alpha)}$, où
$\overline{\Kern(\alpha)}$ désigne l'adhérence dans $V_{T}$ de $\Kern(\alpha)$ avec $\alpha$ une racine. Considérons le schéma:
\begin{center}
$\Omega:=\prod\limits_{V_{T}/\kc}(T\times V_{T})$.
\end{center}
Pour tout $S$-schéma sur $\kc$, les $S$-points de $\Omega$ sont donnés par:
\begin{center}
$\Omega(S)=\Hom_{V_{T}}(S\times_{\kc}V_{T}, T\times V_{T})=\theta_{*}(T\times V_{T})$.
\end{center}
Le morphisme $\theta:V_{T}\rightarrow\kc$ étant fini plat, nous obtenons que $\Omega$ est représentable. C'est un schéma en groupes lisses et commutatifs de dimension $r\left|\bw\right|$ et au-dessus de l'ouvert régulier semi-simple, $\theta$ étant fini étale, $\Omega$ restreint à cet ouvert est un tore. L'action diagonale de $W$ sur $T\times V_{T}$ induit une action sur $\Omega$.
On considère alors:
\begin{equation}
J^{1}=\Omega^{W}=(\prod\limits_{V_{T}/\kc}(T\times V_{T}))^{W}
\label{jun}
\end{equation}
Comme la caractéristique du corps est première avec l'ordre de $W$, on a que $J^{1}$ est un schéma en groupes lisse sur $\kc$.

\begin{prop}
\label{centralisateur}
On a un morphisme canonique $J\rightarrow J^{1}$ qui est un isomorphisme au-dessus de $\kc^{rs}$.

\end{prop}

\begin{proof}
On commence par construire un morphisme de $J$ dans $\Omega$; par adjonction, il revient au même de construire une flèche
\begin{center}
$\theta^{*}J\rightarrow T\times V_{T}$
\end{center}
au-dessus de $V_{T}$.
Nous voulons construire une flèche $V_{B}\rightarrow V_{T}$, il résulte du principe d'extension de Renner \cite[Thm 5.2]{Re2}, qu'il nous suffit de construire une flèche $f: B_{+}\rightarrow T_{+}$ et une flèche $g:V_{T}\rightarrow V_{T}$ telle que $g_{\vert T_{+}}=f$.
On considère alors la projection $B_{+}\rightarrow T_{+}$ et l'application identité $\Id: V_{T}\rightarrow V_{T}$ et l'on obtient une flèche: 
\begin{equation}
\phi:V_{B}\rightarrow V_{T}.
\label{projvin}
\end{equation}
On rappelle que $\tilde{V}_{G}=\{(g,\g)\in G/B\vert ~g^{-1}\g g \in V_{B}\}$. On définit alors une application $\tilde{V}_{G}\rightarrow V_{T}$ donnée par:
\begin{center}
$(g,\g)\rightarrow \phi(g^{-1}\g g)$
\end{center}
Considérons le diagramme commutatif suivant:
$$\xymatrix{\tilde{V}_{G}^{reg}\ar[r]\ar[d]_{\la}&V_{T}\ar[d]^{\theta}\\V_{G}^{reg}\ar[r]&\kc}$$
Ce diagramme est en fait cartésien. En effet, la flèche $\chi_{+}^{reg}$ est lisse, donc $V_{G}^{reg}\times_{\kc}V_{T}$ est lisse au-dessus de $V_{T}$, donc normal.
La flèche $\iota:\tilde{V}_{G}^{reg}\rightarrow V_{G}^{reg}\times_{\kc}V_{T}$ est finie car $\tilde{V}_{G}^{reg}\rightarrow V_{G}^{reg}$ l'est, d'après \cite[Cor. 2.15]{Bt}. De plus, $\iota$ est birationnelle car c'est un isomorphisme au-dessus de $G_{+}^{rs}$, il résulte alors du Main Theorem de Zariski que $\iota$ est un isomorphisme.
Il résulte alors de l'isomorphisme entre $\chi_{+}^{*}J$ et $I_{\scriptscriptstyle{\vert V_{G}^{reg}}}$, qu'il nous suffit de construire un morphisme :
\begin{center}
$\lambda^{*}I_{\scriptscriptstyle{\vert V_{G}^{reg}}}\rightarrow T\times \tilde{V}_{G}$.
\end{center}
On commence par rappeler un lemme, de preuve analogue à \cite[Lem. 2.4.3]{N}:
\begin{lem}\label{borelcent}
Pour tout $(gB,\g)\in\tilde{V}_{G}^{reg}$, $I_{\g}\subset \ad(g)B$.
\end{lem}
On considère alors le schéma $\underline{B}$ dont la fibre au-dessus de $gB\in G/B$.
En notant $\underline{B}_{\scriptscriptstyle{\vert \tilde{V}_{G}^{reg}}}$, le changement de base à $\tilde{V}_{G}^{reg}$, du lemme ci-dessus, on obtient une flèche
\begin{center}
$I_{\scriptscriptstyle{\vert \tilde{V}_{G}^{reg}}}\rightarrow\underline{B}_{\scriptscriptstyle{\vert \tilde{V}_{G}^{reg}}}$,
\end{center}
et donc en changeant à nouveau de base, un morphisme:
\begin{center}
$\lambda^{*}I_{\scriptscriptstyle{\vert V_{G}^{reg}}}\rightarrow\underline{{B}}_{\scriptscriptstyle{\vert \tilde{V}_{G}}}$.
\end{center}
Comme de plus, nous avons un morphisme canonique de  $\underline{B}$ vers $T\times G/B$, d'où l'on obtient une flèche $G_{+}$-équivariante entre 
\begin{center}
$\lambda^{*}I_{\scriptscriptstyle{\vert V_{G}^{reg}}}\rightarrow T\times \tilde{V}_{G}$,
\end{center}
qui est un isomorphisme au-dessus du lieu régulier semi-simple.
Au-dessus de $\tilde{V}_{G}^{rs}$, l'action de $W$ sur $I$ se transporte à $T\times \tilde{V}_{G}^{rs}$.
Par adjonction, on a une flèche:
\begin{center}
$I_{\scriptscriptstyle{\vert V_{G}^{reg}}}\rightarrow\lambda_{*}(T\times \tilde{V}_{G})$
\end{center}
qui se factorise par le sous-schéma des points fixes sous $W$.
Par descente, on obtient le morphisme $J\rightarrow J^{1}$ qui est un isomorphisme au-dessus de $\kc^{rs}$.
\end{proof}

Nous allons avoir besoin de considérer un sous-schéma ouvert de $J^{1}$.
\begin{defi}
Soit $\tilde{J}$ le sous-foncteur de $J^{1}$ qui à tout $\kc$-schéma $S$, associe le sous-ensemble $\tilde{J}(S)$ de $J^{1}(S)$ des morphismes $W$-équivariants:
\begin{center}
$f:S\times_{\kc}V_{T}\rightarrow T$
\end{center}
tel que pour tout point géométrique $x$ de $S\times_{\kc}V_{T}$ stable sous une involution $s_{\alpha}(x)=x$ attachée à une racine $\alpha$, on a $\alpha(f(x))\neq -1$.
\end{defi}
\begin{lem}
Le foncteur $\tilde{J}$ est représentable par un sous-schéma ouvert affine de $J^{1}$. De plus, on a les inclusions $J^{0}\subset\tilde{J}\subset J^{1}$.
\end{lem}
\begin{proof}
On renvoie à \cite[Lem. 2.4.6]{N}, la seule différence est qu'en lieu et place des hyperplans de racines $h_{\alpha}$, il faut considérer l'adhérence dans $V_{G}$ du tore $\Kern\alpha$.
\end{proof}

\begin{prop}\label{galois}
Le morphisme $J\rightarrow J^{1}$ se factorise via $\tilde{J}$ et induit un isomorphisme entre $J$ et $\tilde{J}$.
\end{prop}

\begin{proof}
On sait déjà que nous avons la proposition au-dessus de $T_{+}/W$, d'après \cite[2.4.7]{N}.
De plus, on sait qu'au-dessus de $\kc^{rs}$, $J$, $J^{1}$ et $\tilde{J}$ sont tous isomorphes au tore $T$.

En particulier, on obtient la proposition sur l'ouvert $\kc^{rs}\cup T_{+}/W$, lequel est de codimension au moins deux dans $\kc$, car $\kc^{rs}$ est ouvert dans chaque strate.
Comme les schémas $J$ et $\tilde{J}$ sont lisses sur  $\kc$, l'isomorphisme se prolonge, ce qu'on voulait.
\end{proof}
\section{Le champ de Hecke}
Soit $X$ une courbe projective lisse géométriquement connexe sur un corps algébriquement clos $k$. On note $F$ son corps de fonctions, $\left|X\right|$ l'ensemble des points fermés.
Pour $x\in\left|X\right|$, soit $D_{x}=\Spec(\mathcal{O}_{x})$ le disque formel en $x$ et $D_{x}^{\bullet}=\Spec(\fx)$, le disque formel épointé, d'uniformisante $\pi_{x}$.

Soit $\bg$ un groupe connexe réductif déployé tel que $\bg_{der}$ soit simplement connexe. 
On considère une paire de Borel $(\bB,\bt)$ de $\bg$ et on note $\bw$ le groupe de Weyl de $\bg$. On suppose que l'ordre de $\bw$ est premier avec la caractéristique.

Soit un schéma en groupes $G$  connexe réductif sur $X$ qui se déploie en $\bg$ sur un revêtement fini galoisien $\rho:X_{\rho}\rightarrow X$ de groupe $\Gamma$. Cela revient à la donnée d'un morphisme:
\begin{center}
$\rho:\Gamma\rightarrow \out(\bg)$.
\end{center}
Le schéma en groupes $G$ est alors muni d'une paire de Borel $(B,T)$ et on considère $W$ le schéma en groupes fini étale obtenu en tordant $W$ par $\rho$.
On note $X_{*}(\bt)^{+}$ l'ensemble des cocaractères dominants de $\bt$.
On désigne par $\Bun$ le champ des  $G$-torseurs sur $X$. Le champ $\Bun$ est algébrique au sens d'Artin (\cite{LM} et \cite[Prop. 1]{H}).

\subsection{Le cas déployé} Nous reprenons la présentation qui en est faite par Frenkel-Ngô \cite{FN}. Les objets considérés ont été introduits par Beilinson-Drinfeld \cite {BD}.

Pour tout point fermé $x\in X$, considérons le $k$-schéma en groupes $K_{x}=\bg(\co_{x})$ et le ind-schéma en groupes $\bg_{x}:=\bg(F_{x})$, on dispose de la grassmannienne affine en $x$, $\Gr_{x}=\bg_{x}/K_{x}$. Elle admet d'après Lusztig une structure d'ind-schéma ainsi qu'une stratification en $K_{x}$-orbites, localement fermées, dite de Cartan, indexée par les cocaractères dominants $X_{*}(T)^{+}$. Pour $\la\in X_{*}(T)^{+}$, on note $\Gr_{\la,x}$ une telle strate et $\overline{\Gr}_{\la,x}$ son adhérence dans $\Gr_{x}$.
Plus généralement, étant donné un ensemble fini de points fermés $S\subset\left|X\right|$, nous notons 
\begin{center}
$\Gr_{S}:=\prod\limits_{s\in S}\Gr_{s}$.
\end{center}
Soit $\Div(X)$ le groupe abélien des diviseurs sur $X$. On considère alors le groupe $\Div(X,T)=\Div(X)\otimes_{\mathbb{Z}}X_{*}(T)$ des diviseurs à coefficients dans les cocaractères de $T$. On note $\Div^{+}(X,T)\subset\Div(X,T)$ le cône formé par les diviseurs dont les coefficients sont dans $X_{*}(T)^{+}$. On a alors une action de $W$ sur $\Div(X,T)$ induite par l'action de $W$ sur $X_{*}(T)$.
Pour $\la=\sum\limits_{x\in \left|X\right|}\la_{x}[x]\in \Div(X,T)$, on pose $S=\supp(\la):=\{x\in \left|X\right|\vert ~\la_{x}\neq 0\}$ et nous notons $\Gr_{\la}$ (resp. $\overline{\Gr}_{\la}$) le produit 
\begin{center}
$\prod\limits_{s\in S}\Gr_{\la_{s},s}$
\end{center}
(resp. $\prod\limits_{s\in S}\overline{\Gr}_{\la_{s},s}$).
\begin{defi}
On définit le champ de Hecke $\mathcal{H}_{S}$, dont le groupoïde des $R$-points $\mathcal{H}_{S}(R)$, pour une $k$-algèbre $R$, est constitué des triplets $(E, E', \beta)$ où $E, E'$ sont des $\bg$-torseurs sur $X_{R}:=X\times_{k}R$ et $\beta$ un isomorphisme:
\begin{center}
$E_{\scriptscriptstyle{\vert X_{R}-\Gamma_{S,R}}}\stackrel{\cong}{\rightarrow}E'_{\scriptscriptstyle{\vert X_{R}-\Gamma_{S,R}}}$.
\end{center}
où $\Gamma_{S, R}$ désigne le graphe des points $x$ de $S$ dans $X_{R}$.
\end{defi}

Nous avons une flèche 
\begin{center}
$\inv:\cH_{S}\rightarrow[K_{S}\backslash\Gr_{S}]$
\end{center}
qui associe à un triplet $(E,E',\beta)$ la position relative du triplet local $(E_{S},E_{S'},\beta_{\vert F_{S}})$, où $E_{S}$ est la restriction de $E$ à $D_{S}$. Pour $\la=\sum\limits_{s\in S}\la_{s}[s]\in \Div^{+}(X,T)$, on note $\inv_{S}(E,E',\beta)=\la$ si la paire $(E,E')$ est en position relative $\la$.

\begin{defi}
Pour $\la=\sum\limits_{s\in S}\la_{s}[s]\in \Div^{+}(X,T)$, on considère alors le sous-champ fermé $\overline{\cH}_{\la}$, qui est l'image réciproque du fermé $\overline{\Gr}_{\la}$ par la flèche $\inv$.
\end{defi} 
$\rmq$ D'après le lemme 3.1 de Varshavsky \cite{Va}, il résulte que $\overline{\cH}_{\la}$ est un champ algébrique localement de type fini sur $k$.
Nous en déduisons donc que $\cH$ a une structure de ind-champ algébrique.
Nous avons deux projections $p_{1}$ et $p_{2}$:
$$\xymatrix{&\mathcal{H}_{S}\ar[dr]^{p_{2}}\ar[dl]_{p_{1}}\\\Bunb&&\Bunb}$$
où $p_{1}(E, E', \beta)=E$ et  $p_{2}(E, E', \beta)=E'$.
Les fibres de $p_{1}$ et $p_{2}$ sont des formes tordues de la grassmannienne affine $\Gr_{S}$.
De plus, le sous-champ fermé $\overline{\cH}_{\la}$ est fibré au-dessus de $\Bunb$ en $\overline{\Gr}_{\la}$.
Nous avons la proposition suivante due à Varshavsky \cite[A8c]{Va}, qui nous dit que pour la topologie lisse, la fibration est localement triviale:
\begin{prop}\label{va}
Il existe un morphisme $V\rightarrow \Bunb$ lisse à fibres géométriquement connexes tel que $\overline{\cH}_{\la}\times_{\Bunb}V$ est isomorphe à $V\times_{k}\overline{\Gr}_{\la}$, le produit fibré se faisant indifféremment pour $p_{1}$ ou $p_{2}$. 
\end{prop}
En particulier, les projections sont plates, projectives et algébriques au-dessus de $\Bunb$.
\subsection{Le cas quasi-déployé}\label{quasidep}
La définition du champ de Hecke dans le cas quasi-déployé est due à Frenkel-Ngô \cite{FN} et la définition de la grassmannienne affine dans le cas quasi-déployé à Heinloth \cite{H}.
On rappelle que nous avons un revêtement fini galoisien de groupe $\Gamma$:
\begin{center}
$\rho:X_{\rho}\rightarrow X$
\end{center}
qui déploie $G$ en $\bg$ ainsi que le schéma en tores $T$ en $\bt$.
Le schéma en groupes $G$ s'identifie alors à 
\begin{center}
$\bg\times^{\Gamma}X_{\rho}$.
\end{center}
En particulier un $G$-torseur sur $X$ est un $\bg$-torseur sur $X_{\rho}$, $\Gamma$-équivariant, i.e. pour tout $\g\in\Gamma$, on a un isomorphisme 
\begin{center}
$i_{\g}:\g^{*}(\tilde{E})\rightarrow\tilde{E}$
\end{center}
où $\g^{*}$ est induit à la fois par l'action de $\g$ sur $X_{\rho}$ et $\bg$ avec une relation de compatibilité $i_{\g_{1}}i_{\g_{2}}=i_{\g_{1}\g_{2}}$ pour $\g_{1},\g_{2}\in\Gamma$.

\begin{defi}
Soit $\cH_{S}$ le ind-champ classifiant les quadruplets $(\tilde{E},\tilde{E'},\beta')$, $\tilde{E}$ et $\tilde{E'}$ des $\bg$-torseurs $\Gamma$-équivariants sur $X_{\rho}$, $\beta'$ un isomorphisme en dehors de $\rho^{-1}(S)$, invariant sous l'action de $\Gamma$ sur $\tilde{E}$ et $\tilde{E'}$.
\end{defi}
Pour chaque point géométrique $x$ de $S$, la restriction de $G$ à $D_{x}$ est isomorphe à $\bg$.
On obtient donc un isomorphisme entre $G_{\vert D_{x}}$ et le schéma en groupes constant $\bg\times D_{x}$, bien défini modulo l'action de $\Gamma$. 
En particulier, on obtient que le morphisme $\cH_{S}\rightarrow\Bun$ a ses fibres géométriques isomorphes à la grassmannienne affine $\Gr_{\bg}$.

De plus, si $E$ et $E'$ sont des $G$-torseurs sur $D_{x}$, muni d'un isomorphisme sur $D_{x}^{\bullet}$, alors $\inv_{x}(E,E')$ est bien défini comme une orbite $[\la]\in[X_{*}(\bt)^{+}/\Gamma]$.
Ainsi, pour $\la=\sum\limits_{x\in X}[\la_{x}][x]$, où $[\la_{x}]$ désigne une orbite $[X_{*}(\bt)^{+}/\Gamma]$, on peut définir le sous-champ fermé $\overline{\cH}_{\la}$ du champ de Hecke $\cH_{S}$.
Il est fibré au-dessus de $\Bun$ en l'union $\overline{\Gr}_{\la}$ pour $\la$ dans la $\Gamma$-orbite de $[\la]$.

Dans la suite, on note également $\Div^{+}(X,T)$ l'ensemble des combinaisons formelles $\la=\sum\limits_{x\in X}[\la_{x}][x]$ avec $[\la_{x}]\in [X_{*}(\bt)^{+}/\Gamma]$.
\section{La fibration de Hitchin}
Soit $S$ un ensemble fini de $\left|X\right|$ et $\la=\sum\limits_{s\in S}[\la_{s}][s]$.
On considère le carré cartésien suivant:
$$\xymatrix{\mathcal{M}_{S}\ar[d]\ar[r]&\cH_{S}\ar[d]\\\Bun\ar[r]^-{\Delta}&\Bun\times\Bun}$$
On obtient que $\cm_{S}(k)$ est le groupoïde des couples $(E,\phi)$, où $E$ un $G$-torseur sur $X$ et
$\phi$ une section sur $X-S$ du fibré adjoint
\begin{center}
$\Ad(E)=E\times^{G}G$
\end{center}
où $G$ agit sur lui-même par conjugaison, ou si l'on préfère un automorphisme de $E$ au-dessus de $X-S$. Cet espace a été introduit dans l'article de Frenkel-Ngô \cite[sect. 4.1]{FN}.
Nous obtenons alors par changement de base le sous-champ $\cmd$ à partir du champ $\overline{\cH}_{\la}$.

Dans le cas déployé, on a la description adélique suivante: 
\begin{center}
$\cmd(k):=\bg(F)\backslash\{(\g,(g_{x}))\in \bg(F)\times \bg(\ab)/\bg(\co_{\ab})\vert~ g_{x}^{-1}\g g_{x}\in\overline{K_{x}\pi_{x}^{\la}K_{x}}\}$,
\end{center}
où $\bg(F)$ agit par $h.(\g,(g_{x}))=((h\g h^{-1},(hg_{x})))$.
\bigskip

Nous obtenons également, toujours dans le cas déployé, une interprétation modulaire de $\cmd$ tirée de \cite[sect. 4.5]{FN}. 
\begin{lem}
Pour une $k$-algèbre $R$, le groupoïde $\cmd(R)$ est donné par les uplets $(E,\beta)$ où $E$ est un $\bg$-torseur sur $X_{R}$, $\beta$ un automorphisme de $E$ sur $X_{R}-\Gamma_{S,R}$ tel que pour toute représentation irréductible $\rho_{\mu}$ de plus haut poids $\mu$ $\rho_{\mu}(\beta)$ se prolonge en une injection de fibrés vectoriels :
\begin{center}
$\rho_{\mu}(\beta):\rho_{\mu}(E)\rightarrow\rho_{\mu}(E)(\left\langle \mu,-w_{0}\la\right\rangle)$.
\end{center}
où l'on a poussé le $\bg$-torseur $E$ en un fibré vectoriel par $\rho_{\mu}(E)$. 
\end{lem}

Le lien avec le semi-groupe de Vinberg apparaît grâce au lemme suivant:
\begin{lem}\label{cont} 
Un élément $g\in \bg(F_{x})$ appartient à l'orbite $K\pi_{x}^{\lambda}K$ (resp. $\overline{K_{x}\pi_{x}^{\la}K_{x}}$) si et seulement si pour tout cocaractère dominant $\omega\in X^{*}(\bt)^{+}$, le plus grand des ordres des pôles des coefficients de la matrice $\rho_{\omega}(g)$ est égal à $\left\langle \omega,-w_{0}\lambda\right\rangle$ (resp. inférieur ou égal et égal pour les représentations de dimension un) où $w_{0}$ est l'élément long du groupe de Weyl. De plus, l'élément $g_{+}=(\pi^{-w_{0}\la},g)$ est dans $V_{\bg}^{0}(\mathcal{O}_{x})$ (resp. $V_{\bg}(\mathcal{O}_{x})$).
\end{lem}
\begin{proof}
Le plus grand des ordres des pôles est invariant à gauche et à droite par $K_{x}$, en particulier il nous suffit de regarder celui de $\pi_{x}^{\la}$, cet ordre est égal à $\left\langle \omega,-w_{0}\lambda\right\rangle$. Inversement, les entiers $\left\langle \omega,-w_{0}\lambda\right\rangle$ déterminent uniquement $\lambda$.

Enfin, comme l'élément $g_{+}$  est dans $\bg_{+}(F_{x})$ et que pour tout $\omega\in X^{*}(\bt)^{+}$, $\rho_{\omega}(g_{+})\in \End V_{\omega}(\mathcal{O}_{x})$, la continuité nous donne le résultat voulu.
La preuve pour $\overline{K_{x}\pi_{x}^{\la}K_{x}}$ est analogue.
\end{proof}

\subsubsection{Construction d'un $T$-torseur}
On commence par traiter le cas déployé.
Choisissons une base $\omega_{1}',\dots,\omega_{l}'$ du réseau des caractères
$X^{*}(\bt)$ telle que
\begin{center}
$ \forall~\alpha\in\Delta=\{\alpha_{1},\dots,\alpha_{l}\}, \left\langle \chi,\check{\alpha}\right\rangle=0$
\end{center}
Chaque $\omega_{i'}':\bt\rightarrow\mathbb{G}_{m}$, $1\leq i'\leq l$, se prolonge de manière unique en un caractère
\begin{center}
$\omega_{i'}':\bg\rightarrow\mathbb{G}_{m}$.
\end{center}
Pour tout indice $i', 1\leq i'\leq l$, notons $V_{\omega_{i}'}$ un espace vectoriel de dimension un sur lequel $\bg$ agit par $\omega_{i'}':\bg\rightarrow\mathbb{G}_{m}$.

L'ensemble des poids dominants est stable par translation par les éléments du réseau $\mathbb{Z}\omega_{1}'+\dots+\mathbb{Z}\omega_{l}'$. Le quotient par ce réseau est un cône saturé non dégénéré de $X^{*}(\bt)/(\mathbb{Z}\omega_{1}'+\dots+\mathbb{Z}\omega_{l}')$ qui est engendré par $\bar{\omega}_{1},\dots,\bar{\omega}_{r}$ que l'on relève en une famille de caractères $\omega_{1},\dots,\omega_{r}$ de $\bt$.
Pour tout $\omega\in X^{*}(\bt)$ et $\mu\in\Div^{+}(X,\bt)$, l'accouplement
\begin{center}
$\left\langle\omega, \mu\right\rangle=\sum\limits_{x\in\left|X\right|}\left\langle\omega, \mu_{x}\right\rangle [x]$
\end{center}
définit un diviseur effectif sur la courbe.
On note $\left|\left\langle\omega, \mu \right\rangle\right|$ le degré de ce diviseur.

Pour $x\in \left|X\right|$, on rappelle que $X_{*}(T)=\bt(F_{x})/\bt(\co_{x})$, en particulier pour tout $\mu\in X_{*}(\bt)^{+}$, on peut lui associer un $\bt$-torseur sur $D_{x}$ avec une trivialisation générique.
Pour $\mu=\sum\limits_{s\in S}\mu_{s}[s]\in\Div^{+}(X,\bt)$, considérons alors le $\bt$-torseur $E_{\bt}(\mu)$ sur $X$, qui consiste à recoller le $\bt$-torseur trivial en dehors de $S$ avec le $\bt$-torseur sur le voisinage formel des points de $S$, donné par $\mu$, lequel est muni canoniquement d'une trivialisation générique.
Ce recollement nous est donné par une généralisation de Beauville-Laszlo par Beilinson-Drinfeld \cite[sect. 2.3.7]{BD}.
Pour chaque racine positive $\alpha\in R^{+}$,  nous avons
\begin{center}
$E_{\bt}(\mu)\times^{\bt,\alpha}\mathbb{G}_{m}=\mathcal{O}_{X}(\left\langle\alpha, \mu\right\rangle)$.
\end{center}
Comme $\left\langle \alpha, \mu\right\rangle\geq 0$, on a une section canonique définie sur la courbe que l'on note $1_{\left\langle \alpha, \mu\right\rangle}$.
En se limitant aux racines simples $(\alpha_{i})_{1\leq i \leq r}$, nous obtenons  $r$-fibrés en droites munis de $r$-sections, d'où l'on déduit un morphisme 
\begin{center}
$\mu:X\rightarrow[A_{\bg}/\bz_{+}]$.
\end{center}
où l'on rappelle que $A_{\bg}$ désigne la base du morphisme d'abélianisation 
\begin{center}
$\alpha:V_{\bg}\rightarrow A_{\bg}:=\mathbb{A}^{r}$
\end{center}
donné par les racines simples de $\bg$. 
\medskip

Montrons comment l'on passe au cas quasi-déployé.
On veut construire une flèche 
\begin{center}
$[\mu]:X\rightarrow[A_{G}/Z_{+}]$
\end{center}
à partir d'une combinaison formelle $\mu=\sum\limits_{x\in X}[\mu_{x}][x]\in \Div^{+}(X,T)$.
La donnée d'une telle flèche revient à la donnée d'un $T$-torseur $E_{T}(\mu)$ et d'une section $\phi\in H^{0}(X,A_{G}\wedge^{T}E_{T}(\mu))$.

Le $T$-torseur $E_{T}(\mu)$ s'obtient de la même manière que dans le cas déployé, comme nous avons l'isomorphisme suivant \cite[Thm.5.1]{Rap}:
\begin{center}
$T(F_{x})/T(\co_{x})=X_{*}(\bt)_{\Gamma}$
\end{center}
et comme l'ensemble des orbites $[X_{*}(\bt)^{+}/\Gamma]$ s'identifie à un sous-ensemble des coinvariants $X_{*}(\bt)_{\Gamma}$.
Il nous faut maintenant obtenir une section $H^{0}(X,A_{G}\wedge^{T}E_{T}(\mu))$.
Le $T$-torseur $E_{T}(\mu)$ est un $\bt$-torseur $E'$ sur $X_{\rho}$, $\Gamma$-équivariant et une section de $H^{0}(X,A_{G}\wedge^{T}E_{T}(\mu))$ correspond à la donnée d'une section $\phi'\in H^{0}(X_{\rho},A_{\bg}\wedge^{\bt}E')$ qui est $\Gamma$-invariante pour l'action de $\Gamma$ sur $A_{\bg}\wedge^{\bt}E'$.

Si l'on pousse par une racine simple $\alpha$ ce torseur, on obtient alors le fibré en droites
\begin{center}
 $\co_{X_{\rho}}(\left\langle\alpha,\sum\limits_{y\vert~\rho(y)=x}(\sum\limits_{\sigma\in \Gamma}\sigma\mu_{x})[y] \right\rangle)$.
\end{center}
Enfin, par construction, on a une trivialisation du torseur $E_{T}([\mu])$ en dehors de $S$ qui revient à une trivialisation de $E'$ en dehors de $\rho^{-1}(S)$, $\Gamma$-invariante, laquelle se prolonge l'action de Galois préserve le caractère dominant. On obtient donc la section voulue et donc la flèche:
\begin{center}
$[\mu]:X\rightarrow[A_{G}/Z_{+}]$
\end{center}

\subsubsection{Une autre définition de $\cmd$}\label{autdef2}
On fixe $\la=\sum\limits_{s\in S}\la_{s}[s]\in \Div^{+}(X,T)$. Nous avons vu que $\Div(X,T)$ est muni d'une action du groupe de Weyl $W$, on considère alors l'élément $-w_{0}\la\in \Div^{+}(X,T)$ où $w_{0}$ est l'élément long du groupe de Weyl.
En vertu de la section précédente, nous obtenons une flèche:
\begin{center}
$-w_{0}\la:X\rightarrow[A_{G}/Z_{+}]$.
\end{center}
Nous avons le morphisme d'abélianisation $\alpha:V_{G}\rightarrow A_{G}$ qui  est équivariant par rapport à l'action du centre $Z_{+}$ de $G_{+}$, d'où l'on obtient une flèche :
\begin{center}
$\alpha: [V_{G}/Z_{+}]\rightarrow [A_{G}/Z_{+}]$.
\end{center}
Posons  alors $V_{G}^{\la}:=(-w_{0}\la)^{*}[V_{G}/Z_{+}]$. Nous obtenons un espace fibré sur $X$.
Le semi-groupe de Vinberg admet un ouvert lisse $V_{G}^{0}$ (cf. sect.\ref{introsemi}) et on note de la même manière $V_{G}^{\la,0}$ l'espace tiré sur $X$ par la flèche $-w_{0}\la$.
Nous obtenons maintenant la caractérisation alternative de $\cmd$:
\begin{prop}\label{hitchin}
L'espace de Hitchin $\cmd$ se réinterprète comme le champ des sections 
\begin{center}
$\Hom_{X}(X, [V_{G}^{\la}/G])$
\end{center}
où $G$ agit sur $V_{G}^{\la}$ par conjugaison. 
Il classifie les couples $(E,\phi)$ avec $E$ un $G$-torseur sur $X$ et $\phi$ une section de l'espace fibré au-dessus de $X$
\begin{center}
$V_{G}^{\la}\times^{G}E$.
\end{center}
Nous avons également l'ouvert $\cmdo$ qui classifie les sections
\begin{center}
$h_{(E,\phi)}:X\rightarrow [V_{G}^{\la,0}/G]$.
\end{center}
\end{prop}
\begin{proof}
Pour simplifier, on suppose de plus $G$ déployé.
On construit une flèche de $\cmd$ vers $\Hom (X, [V_{G}^{\la}/G])$. Soit $R$ une $k$-algèbre et $(E,\beta)\in \cmd(R)$. La section $\beta$ est un automorphisme de $E$ en dehors de $X_{R}-\bigcup\limits_{s\in S}\Gamma_{s\times\Spec R}$ qui se prolonge en une injection de fibrés vectoriels sur $X_{R}$.
Fixons $s\in S$, regardons la section autour du voisinage formel épointé $D_{s}^{\bullet}\hat{\times}R$, quitte à localiser sur $R$, on peut supposer que le torseur $E_{D_{s}\hat{\times}R}$ est trivial. En particulier, $\beta$ induit pour toute représentation irréductible $(\rho_{\omega},V_{\omega})$ de plus haut poids $\omega$, une injection de fibrés vectoriels:
\begin{center}
$\rho_{\omega}(\beta):V_{\omega}\rightarrow V_{\omega}(\left\langle\omega,-w_{0}\la \right\rangle)$
\end{center}
et donc en multipliant par $\pi_{S}^{-w_{0}\la}$, on obtient un point $\phi_{S}\in V_{G}^{\la}(R[[\pi_{S}]])$ tel que $\rho_{\omega}(\phi_{S})=\pi_{S}^{-w_{0}\la}\rho_{\omega}(\beta)$. Il ne nous reste plus qu'à recoller $\phi_{S}$ avec $\beta_{X_{R}-\bigcup\limits_{s\in S}\Gamma_{s\times\Spec R}}$ pour obtenir une section $\phi'\in H^{0}(X_{R},V_{G}^{\la})$, d'où une flèche
\begin{center}
$\cmd\rightarrow\Hom (X, [V_{G}^{\la}/G])$.
\end{center}
Nous allons maintenant construire une flèche dans le sens inverse dont on laissera au lecteur le soin de vérifier qu'elles sont réciproques l'une de l'autre. Étant donné une paire $(E,\phi)$ avec une section $\phi\in H^{0}(X_{R},V_{G}^{\la}\times^{G}E)$, en dehors du graphe des points de $S$, on a $\la_{x}=0$, donc  $\phi$ est un automorphisme de $E_{\vert X_{R}-\Gamma_{S,\Spec R}}$ et localement en les points de $S$, comme $\phi\in V_{G}^{\la_{s}}(R[[\pi_{s}]])$, sa partie abélienne est donné par $\pi_{s}^{-w_{0}\la}$, et donc en multipliant par $\pi_{s}^{-w_{0}\la}$, nous obtenons notre injection au niveau des fibrés vectoriels qui prolonge la section sur $X_{R}-\Gamma_{S,\Spec R}$.
\end{proof}
\begin{defi}
On définit également l'ouvert régulier $\cmdo^{reg}$ de $\cmdo$ par le champ des sections:
\begin{center}
$\cmdo^{reg}:=\Hom_{X}(X,[V_{G}^{\la,reg}/G])$.
\end{center}
avec $V_{G}^{\la,reg}=(-w_{0}\la)^{*}[V_{G}^{reg}/Z_{+}]$ et où $V_{G}^{reg}\subset V_{G}$ désigne le lieu où la dimension du centralisateur est minimale.
\end{defi}
Passons à la définition de la base de cette fibration.
Nous avons un morphisme de Steinberg 
\begin{center}
$\chi_{+}: V_{G}\rightarrow\kc$
\end{center}
ainsi qu'un morphisme de projection $p_{1}:[\kc/Z_{+}]\rightarrow[A_{G}/Z_{+}]$ qui est lisse de dimension relative $r$ toujours d'après le théorème \ref{bouth}. On note $\kc^{\lambda}$ le changement de base à $X$. 
On définit alors la base de Hitchin $\abd$ comme le champ des sections
\begin{center}
$h_{a}:X\rightarrow\kc^{\la}$.
\end{center}
Grâce au morphisme de Steinberg $\chi_{+}$, nous avons un morphisme de Hitchin 
\begin{center}
$f:\cmd\rightarrow\abd$
\end{center}
donné par $f(E,\phi)=\chi_{+}(\phi)$.

$\rmq$ A ce stade, il est important remarquer que $\cmd$ peut être vide si nous n'imposons pas de conditions sur $\la$.
Le cas déployé est suffisant pour l'explication qui suit.
Nous avons une suite exacte de la forme:
$$\xymatrix{1\ar[r]&\bg_{der}\ar[r]&\bg\ar[r]^{\det_{\bg}}&\mathbb{G}_{m}^{l}\ar[r]&1}$$

En particulier, étant donné un point $(E,\phi)\in\cmd(k)$ en considérant $\det_{\bg}(\phi)$, cela nous fournit $l$-fonctions sur la courbe $X$ dont le degré en vertu du lemme \ref{cont}, est $\deg\left\langle \omega_{i}',-w_{0}\la\right\rangle$.
Cela impose donc que:
\begin{center}
$\forall~ 1\leq i\leq l,\deg\left\langle \omega_{i}',-w_{0}\la\right\rangle=0$,
\end{center}
hypothèse que nous ferons par la suite systématiquement.
Dans ce cas, la base de Hitchin pour le groupe $\bg$ est donné par :
\begin{center}
$\abd=\bigoplus\limits_{j=1}^{l}(H^{0}(X,\co_{X}(\left\langle \omega_{j}',-w_{0}\la\right\rangle))-\{0\})\oplus\bigoplus\limits_{i=1}^{r}H^{0}(X,\co_{X}(\left\langle \omega_{i},-w_{0}\la\right\rangle))$.
\end{center}

En particulier, une condition nécessaire pour que $\cmd$ soit non vide est que de plus:
\begin{center}
$\forall 1\leq j\leq l, H^{0}(X,\co_{X}(\left\langle \omega_{j}',-w_{0}\la\right\rangle))\neq \{0\}$.
\end{center}

\subsection{Le champ de Picard}
Le centralisateur régulier va nous permettre d'obtenir une action d'un champ de Picard sur les fibres de Hitchin.
Nous avons construit une section de Steinberg $\eps_{+}$ et un centralisateur régulier $J$ qui est un schéma en groupes commutatif et lisse muni d'un morphisme :
\begin{center}
$\chi_{+}^{*}J\rightarrow I$
\end{center}
qui est un isomorphisme sur l'ouvert $V_{G}^{reg}$ constitué des éléments conjugués à $\eps_{+}(\kc)$ et $I$ le schéma des centralisateurs des éléments de $V_{G}$ dans $G$.
On tire alors le centralisateur régulier $J$ en un schéma $J^{\la}:=(-w_{0}\la)^{*}J$ sur $\kcd$.
Pour tout $S$-point de $\abd$, on a une flèche $h_{a}:X\times S\rightarrow\kcd$. Posons $J_{a}:=h_{a}^{*}J^{\la}$ l'image réciproque de $J^{\la}$ sur $\kcd$.
\begin{defi}
On considère le groupoïde de Picard $\mathcal{P}_{a}(S)$ des $J_{a}$-torseurs sur $X\times S$. Quand $a$ varie, cela définit un groupoïde de Picard $\mathcal{P}$ au-dessus de $\abd$.
\end{defi}
La flèche $\chi_{+}^{*}J\rightarrow I$ induit pour tout $S$-point $(E,\phi)$ au-dessus de $a$ une flèche
\begin{center}
$J_{a}\rightarrow \Aut_{X\times S}(E,\phi)=h_{E,\phi}^{*}I$
\end{center}
avec $I$ le schéma des centralisateurs des éléments de $V_{G}$ dans $G$.
On en déduit alors une action du groupoïde de Picard $\mathcal{P}_{a}(S)$ sur le groupoïde $\overline{\mathcal{M}}_{\la}(a)(S)$, et donc une action de $\mathcal{P}$ sur $\cmd$.
De même que pour les fibres de Springer affines \cite[Prop. 3.6]{Bt}, l'orbite régulière  est une gerbe sous $\cP$.

\begin{prop}\label{picardtorseur}
L'ouvert $\cmdo^{reg}$  est une gerbe neutre sous l'action de $\cP$.
\end{prop}
\begin{proof}
La preuve est la même que pour les fibres de Springer affines \cite[Prop. 3.6]{Bt} et la gerbe est neutralisée par la section de Steinberg $\epsilon_{+}$.
\end{proof} 
\subsection{Les ouverts $\abdh$ et $\abdd$}
Nous avons besoin d'introduire des ouverts de $\abd$ pour lesquels nous avons plus de prise sur la fibration de Hitchin. Ils seront étudiés de manière approfondie dans le chapitre suivant.
On rappelle que nous avons un morphisme fini plat $W$-équivariant:
\begin{center}
$V_{T}\rightarrow\kc$.
\end{center}
ramifié le long du diviseur discriminant $\mathfrak{D}_{+}=\bigcup\limits_{\alpha\in R}\overline{\Kern(\alpha)}$. On note toujours $\mathfrak{D}_{+}\subset\kc$ l'image dans $\kc$ que l'on tire ensuite sur $\kcd$ en un diviseur $\mathfrak{D}_{\la}:=(-w_{0}\la)^{*}\mathfrak{D}$ (on tire de même $V_{T}$ en $V_{T}^{\la}$).
\begin{defi}
On définit l'ouvert génériquement régulier semisimple $\abdh\subset\abd$ constitué des $a\in\abd$ tels que $a(X)\not\subset\mathfrak{D}_{\la}$ ainsi que 
l'ouvert transversal $\abd^{\diamondsuit}$ constitué des $a\in\abd$ qui intersectent transversalement le diviseur discriminant $\mathfrak{D}_{\la}$.
\end{defi}
$\rmq$ En particulier, on a l'inclusion:
\begin{center}
$\abd^{\diamondsuit}\subset\abdh$. 
\end{center}

Nous avons une flèche $\rho:\pi_{1}(X,x)\rightarrow\out(\bg)$ d'image $\Gamma$. 
Soit $W':=W\rtimes\Gamma$, pour $a\in\abdh$, on a une section $h_{a}:X\rightarrow\kcd$. On pose $\tilde{X}_{a}$ le revêtement fini plat obtenu en tirant par $h_{a}$ le revêtement 
\begin{center}
$X_{\rho}\times V_{T}^{\la}\rightarrow\kcd$.
\end{center}
L'ouvert $\abdh$ a la propriété agréable que le champ de Picard est lisse au-dessus de celui-ci.
\begin{prop}\label{picardlisse}
Pour tout $a\in\abdh$,
\begin{center}
$H^{0}(X,\Lie(J_{a}))=\Lie(Z_{\bg})^{\Gamma}$.
\end{center}
Le champ de Picard $\cP^{\heartsuit}:=\cP\times_{\abd}\abdh$ est lisse sur $\abdh$.
Si, de plus $G$ n'a pas de tore déployé central et si $\Gamma$ est d'ordre premier à la caractéristique, alors $\cP_{a}$ est un champ de Picard de Deligne-Mumford.
\end{prop}
\begin{proof}
De la description galoisienne, on déduit que le groupe $H^{0}(X,\Lie(J_{a}))$ s'identifie aux sections $W$-équivariantes:
\begin{center}
$s:\tilde{X}_{a}\rightarrow\mathfrak{t}$.
\end{center}
Comme on a vu que $\tilde{X}_{a}$ est une courbe propre géométriquement connexe et réduite, 
\begin{center}
$H^{0}(\tilde{X}_{a},\mathfrak{t})=\mathfrak{t}$.
\end{center}
et donc en prenant les $W'$-invariants, on déduit de l'annulation $\mathfrak{t}_{der}^{W'}$, que 
\begin{center}
$H^{0}(X,\Lie(J_{a}))=\Lie(Z_{\bg})^{\Gamma}$.
\end{center}
Le schéma $J_{a}$ est un schéma en groupes commutatif et lisse, l'obstruction à déformer un $J_{a}$-torseur est dans le $H^{2}(X,\Lie(J_{a}))$, qui est nul comme nous sommes sur une courbe. En particulier, $\cP_{a}$ est lisse.
Enfin, comme la dimension du $H^{0}$ est constante, on en déduit que la dimension de l'espace tangent reste constante, d'où la lissité de $\cP\times_{\abd}\abdh$ sur $\abdh$.
De plus, nous avons:
\begin{center}
$H^{0}(X,J_{a})=\bt^{\bw\rtimes\Gamma}$
\end{center}
lequel est fini non-ramifié si $G$ n'a pas de tore central et $\Gamma$ premier à la caractéristique et donc sous cette hypothèse $\cP_{a}$ est bien de Deligne-Mumford.
\end{proof}
\begin{cor}\label{lissitereg}
L'ouvert régulier $\cmdo^{reg,\heartsuit}$ est lisse au-dessus de $\abdh$. 
\end{cor}
\begin{proof}
D'après la proposition $\ref{picardtorseur}$, $\cmdo^{reg,\heartsuit}$ est une gerbe sous $\cP^{\heartsuit}$ et comme $\cP^{\heartsuit}$ est lisse sur $\abdh$, on conclut.
\end{proof}
En général, la fibration de Hitchin n'a aucune raison d'être lisse. Néanmoins, la situation est plus agréable au-dessus de l'ouvert transversal $\abdd$.
\begin{prop}\label{fact}
Soit $\cmdo^{reg,\diamond}$ la restriction à $\abdd$ de $\cmdo^{reg}$ et $f^{\diamond}:\cmdo^{reg,\diamond}\rightarrow\abdd$, alors nous avons un carré cartésien :
$$\xymatrix{\cmdo^{reg,\diamond}\ar[d]_{f^{\diamond}}\ar[r]&\cmd\ar[d]^{f}\\\abdd\ar[r]&\abd}$$
en particulier, pour tout $a\in\abdd$, nous avons $\cmd(a)=\cmdo^{reg}(a)$.
\end{prop}
\begin{proof}
On suit la preuve de \cite[Prop. 4.2]{N2}.
Pour démontrer que le carré est cartésien, il nous suffit de voir qu'au-dessus de $\abdd$, nous avons l'égalité $\cmd=\cmdo^{reg}$. En particulier, il suffit de montrer qu'étant donné une section $h_{(E,\phi)}:X\rightarrow[V_{G}^{\la}/G]$ tel que $\chi_{+}(\phi)\in\abdd$, alors la section $h_{(E,\phi)}$ se factorise par $[V_{G}^{\la,reg}/G]$.
Le problème étant local, on peut supposer que $E$ est trivial et se restreindre à $X=\Spec(k[[\pi]])$.
La section revient alors à une flèche:
\begin{center}
$\phi:\Spec(\co_{V_{G}^{\la,x}})\rightarrow\Spec(k[[\pi]])$.
\end{center}
On note $\phi^{\sharp}$ le morphisme entre les anneaux locaux et $\km_{x}$ l'idéal maximal de $\co_{V_{G}^{\la,x}}$.
Par localité de $\phi^{\sharp}$, on a que $(\phi^{\sharp})^{-1}(\pi k[[\pi]])=\km_{x}$.
On doit montrer que $x\in V_{G}^{\la,reg}$. S'il n'est pas régulier, alors:
\begin{center}
$\mathfrak{D}_{\la}\in\km_{x}^{2}$.
\end{center}
En particulier, l'image du discriminant dans $k[[t]]$ est de valuation au moins deux, ce qui contredit le fait que $a\in\abdd$.
\end{proof}
\begin{defi}\label{plusgrand}
Pour un entier $N$ et $\la\in\Div^{+}(X,T)$, on dit que $\la\succ N$ si, pour tout $\omega\in X^{*}(T)^{+}$ non nul, nous avons :
\begin{center}
$\left|\left\langle \omega,-w_{0}\la\right\rangle\right|\geq N$.
\end{center}
\end{defi}

Maintenant, il nous faut s'assurer que l'ouvert $\abdd$ est bien non vide, ce qui fait l'objet de la proposition suivante, de preuve identique à celle de Ngô \cite[Prop. 4.7.1]{N}:
\begin{prop}
Supposons $\la\succ2g$, alors l'ouvert $\abdd$ est non vide.
\end{prop}

\section{Énoncé des théorèmes principaux}

Nous voulons dans cette section calculer le complexe d'intersection de l'espace total de Hitchin $\cmd$. Nous avons vu qu'au-dessus de l'ouvert $\abdd$, la fibration de Hitchin était lisse, nous allons maintenant l'étudier sur un ouvert plus gros que $\abdd$ qui prendra en compte les singularités de l'espace de Hitchin. Cet ouvert sera suffisamment gros pour les applications locales que nous avons en vue.
\subsection{Le théorème de transversalité}
On note $F$ le corps de fonctions de notre courbe projective lisse géométriquement connexe $X$ de genre $g$ définie sur un corps algébriquement clos $k$.
Soit $G$ une forme quasi-déployée de $\bg$ sur $X$, où $\bg$ est connexe réductif avec $\bg_{der}$ simplement connexe.
On suppose que $G$ se déploie sur un revêtement étale galoisien $\rho:X_{\rho}\rightarrow X$  de groupe $\Gamma$ et $G$ n'a pas de facteurs simples de type $A_{2r}$.
 
On considère un diviseur  $\la\in \Div^{+}(X,T)$  et on note $S=\supp(\la)$.
On dispose du diviseur discriminant $\mathfrak{D}_{\la}$ sur $\kcd$.
En considérant la flèche:
\begin{center}
$\ev:X\times\abd\rightarrow\kcd$,
\end{center}
on obtient un fibré en droite en tirant le diviseur $\mathfrak{D}_{\la}$ qui est de degré constant $\delta$.
En particulier, pour tout $a\in\abd$, on pose $\Delta(a):=a^{*}\mathfrak{D}_{\la}$, qui est un diviseur effectif de degré constant $\delta$.
On considère le schéma $X^{(\delta)}:=X^{\delta}/\mathfrak{S}_{\delta}$ qui classifie les diviseurs effectifs $D$ sur $X$ de degré $\delta$, où $\mathfrak{S}_{\delta}$ est le groupe symétrique à $d$ éléments.
D'après la discussion précédente, nous obtenons une flèche :
\begin{center}
$\Phi:\abd\rightarrow X^{(\delta)}$.
\end{center}
donnée par $a\mapsto\Delta(a)$.
Pour un point fermé $x\in X(k)$, on note $\Delta_{x}(a)$ le discriminant local en $x$, i.e. le tiré en arrière sur le disque formel en $x$ et on pose $d_{x}(a)=\val(\Delta_{x}(a))$.
Pour un entier $d\in\mathbb{N}$, nous avons un morphisme:
\begin{center}
$\phi_{d}:X\times X^{(\delta-d-1)}\rightarrow X^{(\delta)}$
\end{center}
qui à une paire $(x,D)$ associe le diviseur $(d+1)[x]+D$.
Cette flèche étant propre, on en déduit que le complémentaire de son image dans $X^{(\delta)}$, 
\begin{center}
$U_{d}:=\{D\in X^{(\delta)}\vert~ \forall ~x\in X, m_{x}(D)\leq d\}$,
\end{center}
est ouvert, où $m_{x}(D)$ désigne la multiplicité de $D$ en $x$.
En prenant l'image réciproque par l'application $\Phi$, nous obtenons donc que:
\begin{center}
$\abdD:=\{a\in\abd\vert~\forall~ x\in X, d_{x}(a)\leq d\}$ 
\end{center}
est ouvert.
Pour tout $a\in\abd$, nous avons une décomposition en somme de diviseurs:
\begin{center}
$\Delta(a)=\Delta_{tr}(a)+\Delta_{sing}(a)$,
\end{center}
où $\Delta_{sing}(a)=\sum\limits_{x\vert~d_{x}(a)\geq 2}d_{x}(a)[x]$.
On considère  la fonction :
$$\begin{array}{ccccc}
d_{sing} & : & \abd & \to & \mathbb{N} \\
 & & a & \mapsto & \deg(\Delta_{sing}(a)) \\
\end{array}.$$
Nous avons alors le lemme suivant:
\begin{lem}\label{scsupsing}
La fonction $d_{sing}$ est semi-continue supérieurement, i.e. pour tout $d\in\mathbb{N}$, le sous-schéma de $\abd$ constitué des $a\in\abd$ tels que $d_{sing}(a)\leq d$ est ouvert.
\end{lem}
\begin{proof}

Soit $d\in\mathbb{N}$, on considère la fonction:
$$\begin{array}{ccccc}
m_{sing} & : & X^{(\delta)} & \to & \mathbb{N} \\
 & & D & \mapsto & \sum\limits_{x\vert m_{x}(D)\geq2}m_{x}(D) \\
\end{array}.$$
Pour obtenir le lemme, il nous suffit de voir que le sous-schéma:
\begin{center}
$V_{d}:=\{D\in X^{(\delta)}\vert~ m_{sing}(D)\leq d\}$
\end{center}
est ouvert, puis de prendre l'image réciproque par $\Phi$.
Le problème étant local sur $X^{(\delta)}$ pour la topologie étale, on peut supposer que $X=\mathbb{A}^{1}$ et dans ce cas $X^{(\delta)}=\{P\in k[t]\vert \deg P=\delta\}$.
Tout polynôme $P\in X^{(\delta)}$ admet une décomposition:
\begin{center}
$P=P_{tranv}P_{sing}$
\end{center}
où $P_{tranv}$ est à racines simples.
La fonction $m_{sing}$ est donnée par:
\begin{center}
$\forall~ P\in X^{(\delta)}, m_{sing}(P)=\deg P_{sing}=\delta-\left|\{\text{racines simples de P}\}\right|$.
\end{center}
Montrons que le nombre de racines simples croît par générisation; il nous faut alors relever une racine simple $x_{s}\in k$ de la réduction $P_{s}\in k[t]$ d'un polynôme $P\in k[[\pi]][t]$ en une racine simple $x\in k[[\pi]][t]$ de $P$, ce qui est précisément le lemme de Hensel.
\end{proof}

Nous allons avoir besoin d'un autre invariant pour définir le bon ouvert.
On rappelle que nous avons $\la=\sum\limits_{x\in X}\la_{x}[x]$ et $S=\supp(\la)$. On fixe un point fermé $t\in X$, tel que $\la_{t}\neq 0$, il va jouer le rôle de point auxiliaire.
On note alors $S_{0}=S-\{t\}$.
Pour chaque point fermé $x\in S_{0}$, nous avons un schéma $V^{\la_{x}}$ au-dessus de $D_{x}=Spec(\co_{x})$ le voisinage formel autour de $x$. Ce schéma est lisse en fibre générique.
Soit alors l'entier 
\begin{center}
$e'_{x}:=e^{Elk}_{V^{\la_{x}}/\co_{x}}$,
\end{center}
où l'on renvoie à la définition $\ref{elkinv}$.
Cet entier mesure la singularité du schéma $V^{\la_{x}}$, il est à noter que si $\la_{x}=0$, alors $e'_{x}=0$, puisqu'à ce moment-là, nous sommes dans le groupe.
On note alors $e_{x}=\max(\left\langle 2\rho,\la_{x}\right\rangle,e'_{x})$ et  $e=\sum\limits_{x\in S_{0}}e_{x}$.
On considère alors l'ouvert suivant:
\begin{defi}\label{bemol}
On rappelle que nous avons fixé un point fermé $t\in X$, tel que $\la_{t}\neq 0$.
Pour un entier $d\in\mathbb{N}$, on définit le sous-schéma $\abdbD\subset\abdD$ constitué des $a\in\abdD$ tels que:
\begin{itemize}
\item
$d_{t}(a)=0$.
\item
$3d_{sing}(a) +(e+2d+1)\left|S_{0}\right|+2g-2\prec\la$.
\end{itemize}
\end{defi}
\begin{prop}
Le sous-schéma $\abdbD\subset\abdD$ est ouvert.
\end{prop}
\begin{proof}
La première condition est clairement ouverte. Pour la deuxième, cela résulte du fait que la fonction  $d_{sing}$ est semi-continue et que $\abdD$ est déjà ouvert.
\end{proof}
$\rmq$\begin{itemize}
\item Cet ouvert $\abdbD$ peut sembler artificiel, mais pour les applications au lemme fondamental il est suffisant. En effet, nous aurons en un point $x$ de la courbe un cocaractère dominant $\la_{x}$ et un discriminant local $d_{x}$; pour globaliser le problème nous aurons juste à prendre le $\la_{t}$ aussi grand que l'on veut, de telle sorte que l'inégalité 
\begin{center}
$3d_{sing}(a) +(e+2d+1)\left|S_{0}\right|+2g-2\prec\la$,
\end{center}
puisse être remplie.
\item
Nous avons pris pour $x\in S_{0}$, $e=\max(\left\langle 2\rho,\la_{x}\right\rangle,e'_{x})$, mais l'on doit pouvoir prouver que nous avons une inégalité:
\begin{center}
$\left\langle 2\rho,\la_{x}\right\rangle\leq e'_{x}$,
\end{center}
mais cela n'est pas indispensable à notre propos.
\end{itemize}
\medskip
On forme alors le carré cartésien :
$$\xymatrix{\cmdbD\ar[r]\ar[d]&\cmd\ar[d]\\\abdbD\ar[r]&\abd}.$$
On pose alors 
\begin{center}
$\cmdb=\bigcup\limits_{d\in\mathbb{N}}\cmdbD$,
\end{center}
le théorème principal est alors le suivant:
\begin{thm}\label{transverse}
Le champ $\cmdb$ est équidimensionnel,  soit $m$ sa codimension, alors on a l'égalité suivante entre les complexes d'intersections:
\begin{center}
$(\Delta^{\flat})^{*}[-m]IC_{\overline{\cH}_{\la}}=IC_{\cmdb}$.
\end{center}
\end{thm}
$\rmq$ Ce théorème va s'obtenir en montrant qu'une certaine flèche vers la grassmannienne affine est lisse. 
\medskip

On rappelle que nous avons $S=\supp(\la)$. D'après \cite[Prop. 1.10-1.14]{VLaf}, si nous avons un diviseur $N=\sum\limits_{i\in S}n_{i}[x_{i}]$ avec des $n_{i}$ suffisamment grands par rapport à $\la$, nous disposons d'une flèche lisse:
\begin{center}
$g:\overline{\cH}_{\la}\rightarrow[\overline{\Gr}_{\la}/G_{N}]$.
\end{center}
où $G_{N}:=\Res_{N/k}G$ est la restriction à la Weil de $G$ à $N$.
Nous obtenons donc une flèche composée:
\begin{center}
$g^{\flat}:\cmdb\rightarrow[\overline{\Gr}_{\la}/G_{N}]$
\end{center}
Dans la définition \ref{bemol}, nous avions fixé un point auxiliaire $t\in X$ tel que $\la_{t}\neq 0$. Comme en ce point, on impose au polynôme caractéristique d'être régulier semisimple, l'image de $f^{\flat}$ va tomber dans l'ouvert 
\begin{center}
$U:=\Gr_{\la_{t}}\times\prod\limits_{s\neq t}\overline{\Gr}_{\la_{s}}$.
\end{center}
Posons $\overline{\cH}'_{\la}:=g^{-1}(U)$, $\overline{\Gr}'_{\la}:=\prod\limits_{s\neq t}\overline{\Gr}_{\la_{s}}$ et $p:U\rightarrow \overline{\Gr}'_{\la}$.
On considère alors la flèche $g^{\flat}$ composée avec $p$:
\begin{center}
$\theta^{\flat}:\cmdb\rightarrow[\overline{\Gr}'_{\la}/G'_{N}]$
\end{center}
avec $G_{N}':=\Res_{N-n_{t}[t]/k}G$.
La proposition est la suivante:
\begin{prop}\label{bemolisse}
La flèche $\theta^{\flat}$ est lisse.
\end{prop}
Nous prouvons cette proposition dans la section \ref{globloc}.

\begin{lem}
Le théorème \ref{transverse} se déduit de la proposition \ref{bemolisse}.
\end{lem}
\begin{proof}
En effet, on a un diagramme commutatif:
$$\xymatrix{\cmdb\ar[r]\ar[dr]_{\theta^{\flat}}&\overline{\cH}'_{\la}\ar[d]^{p\circ g}\\&\overline{\Gr}'_{\la}}$$
où les flèches $\theta^{\flat}$ et $p\circ g$ sont lisses et $\overline{\cH}'_{\la}$ est ouvert dans $\overline{\cH}_{\la}$.
\end{proof}
\subsection{Rappels sur des résultats d'Elkik et Gabber-Ramero}\label{rappel}

Les résultats de ce paragraphe sont tirés d'Elkik \cite{Elk} et de Gabber-Ramero \cite{GR}.  Nous suivons la présentation de Temkin \cite[3.2.1]{Tem} et commençons par des généralités sur les idéaux jacobiens.
\medskip
Soit $A$ un anneau et $Q:=A[X_{1},\dots,X_{N}]$. On considère un idéal de type fini $J\subset Q$  et on pose $B=Q/J$.
Soit $(f_{1},\dots,f_{q})$ un système de générateurs de $J\subset Q$. Pour chaque entier $p$ et chaque multi-indice $\alpha=(\alpha_{1},\dots\alpha_{p})\in\mathbb{N}^{p}$ tel que $1\leq\alpha_{1}<\alpha_{2}<\dots<\alpha_{p}\leq q$, posons $\left|\alpha\right|=p$.
Soit $J_{\alpha}\subset J$ le sous-idéal engendré par $(f_{\alpha_{1}},\dots,f_{\alpha_{p}})$ et $\Delta_{\alpha}$ l'idéal engendré par les déterminants des mineurs d'ordre $p$ de la matrice Jacobienne $(\partial f_{\alpha_{i}}/\partial X_{j}~\vert ~1\leq i\leq p,1\leq j\leq N)$.
Soit également l'idéal :
\begin{center}
$(J_{\alpha}:J):=\{f\in A[X_{1},\dots,X_{N}]\vert~ fJ\subset J_{\alpha}\}$.
\end{center}
On pose alors :
\begin{center}
$H_{B/A}^{Elk}\:=\sum\limits_{p\geq0}\sum\limits_{\left|\alpha\right|=p}\Delta_{\alpha}(J_{\alpha}:J)$
\end{center}
L'inconvénient de cet idéal est qu'il dépend du choix des équations et que cet idéal ne contient pas naturellement $J$. Nous allons donc introduire une version plus intrinsèque tirée de Gabber-Ramero \cite{GR}. 
Avec les mêmes notations que ci-dessus, nous faisons la définition suivante :
\begin{defi}\label{gabber}
Soit l'idéal $\textbf{H}_{Q}:=\Ann_{Q}\Ext_{B}^{1}(L_{B/A},J/J^{2})$ où $L_{B/A}$ est le complexe cotangent défini par Illusie \cite{Ill}. 
\end{defi}
$\rmqs$
\begin{itemize} 
\item
Comme $J\subset\textbf{H}_{Q}$, on peut voir $V(\textbf{H}_{Q})$ comme un sous-schéma de $\Spec(B)$ d'idéal $\textbf{H}_{B/A}=\textbf{H}_{Q}B$.
\item
D'après \cite[5.4.3 (iii)]{GR}, cet idéal est le plus grand idéal qui annule tout module de la forme $\Ext_{B}^{1}(L_{B/A},N)$ pour tout $B$-module $N$ et donc ne dépend que de $B\rightarrow A$.
\end{itemize}

D'après \cite[5.4.2]{GR} nous avons les assertions suivantes :
\begin{prop}\label{gr}
\begin{itemize}
\item
Le lieu d'annulation de $\bh_{B/A}$ est précisément le lieu singulier de $f$. En particulier, si $f$ est lisse, $\bh_{B/A}=B$.
\item
Soit $A\rightarrow A'$ et $B'=B\otimes_{A}A'$ alors $\bh_{B/A}B'\subset\bh_{B'/A'}$
\item
Nous avons la comparaison suivante entre l'idéal d'Elkik et celui de Gabber-Ramero :
\begin{center}
$H_{B/A}^{Elk}\subset\bh_{B/A}$.
\end{center}
\end{itemize}
\end{prop}
$\rmq$ En particulier, on voit que l'idéal de Gabber-Ramero raffine l'idéal d'Elkik de même qu'il grossit par changement de base.
\medskip

Pour tout $a:=(a_{1},\dots,a_{N})\in A^{N}$, soit $\mathfrak{p}_{a}\subset F$ l'idéal engendré par $(X_{1}-a_{1},\dots, X_{N}-a_{N})$.
Pour le lemme suivant, nous considérons $A=R[[\pi]]$ avec un anneau local artinien $R$, d'idéal maximal $\mathfrak{m}$, de corps résiduel $k$, $I$ un idéal de carré nul tel que $\bar{R}=R/I$ et $I.\mathfrak{m}=0$.
Nous conservons les notations  de la définition \ref{gabber}.
\begin{lem}\label{gabinf}
Soient $n,h$ deux entiers positifs avec $n>2h$, $a\in A^{N}$ tel que:
\begin{center}
$\pi^{h}\in (H_{A}(Q,J)+\mathfrak{p}_{a})/\mathfrak{m}$, $J\subset \mathfrak{p}_{a}+\pi^{n}IQ$,
\end{center}
alors il existe $b\in A^{N}$ tel que:
\begin{center}
$b-a\in\pi^{n-h}IA^{N}$ et $J\subset\mathfrak{p}_{b}$.
\end{center}
\end{lem}
$\rmq$ Ce lemme est une version infinitésimale du Lemme 5.4.8 de Gabber Ramero. Il nous indique que seul compte ce qui se passe sur le corps résiduel pour relever, dans le cas d'une situation où $A=R[[\pi]]$ avec $R$ artinien.

\begin{proof}
Nous suivons la preuve de \cite[Lem. 5.4.8]{GR}.

Nous avons un morphisme $\sigma:\Spec A/\pi^{n}I\rightarrow\Spec(B)$, on note $\sigma_{0}$ sa restriction au sous-schéma fermé $\Spec A/\pi^{n-h}I$. Nous souhaitons donc relever $\sigma_{0}$ en $\tilde{\sigma}:\Spec(A)\rightarrow\Spec(B)$.

D'après \cite[3.2.16]{GR}, l'obstruction à l'existence d'un relèvement de $\sigma$ en un morphisme
$\Spec(A)\rightarrow\Spec(B)$ gît dans $\Ext^{1}_{B}(L_{B/A},\pi^{n}I)$. On note $\omega\in\Ext^{1}_{B}(L_{B/A},\pi^{n}I)$ cette obstruction.
Nous avons le diagramme commutatif:
$$\xymatrix{\pi^{n}I\ar[r]^{\alpha}\ar[d]_{\beta}&\pi^{n-h}I\ar[dl]^{\g}\\\pi^{n}I}$$
où $\alpha$ est l'inclusion de $\pi^{n}I\subset\pi^{n-h}I$, $\beta$ est la multiplication par $\pi^{h}$ et $\g$ l'isomorphisme donné par la multiplication par $\pi^{h}$.
Comme la structure de $B$-module est induite par l'extension des scalaires de $\sigma$, nous avons $\mathfrak{p}_{a}.\omega=0$. 
De plus, il résulte de la remarque de la définition \ref{gabber} que $H_{A}(F,J).\omega=0$.
Comme $\pi^{h}\in (H_{A}(Q,J)+\mathfrak{p}_{a})/\mathfrak{m}$, il existe un élément $x\in\mathfrak{m}$ tel que $x+\pi^{h}\in H_{A}(F,J)+\mathfrak{p}_{a}$,
d'où
\begin{center}
$(x+\pi^{h}).\omega=0$.
\end{center}
Or, nous avons $I.\mathfrak{m}=0$ d'où $(x+\pi^{h}).\omega=\pi^{h}\omega=\Ext^{1}_{B}(L_{B/A},\beta)(\omega)=0$ et donc en particulier
$\Ext^{1}_{B}(L_{B/A},\alpha)(\omega)=0$. Comme cette classe est précisément l'obstruction à l'existence de $\tilde{\sigma}$, nous concluons.
\end{proof}
$\rmq$ La preuve du lemme montre qu'il suffit de considérer l'idéal $ H'\supset H_{B/A}$ qui est le plus grand idéal qui annule les modules $\Ext^{1}(L_{B/A},N)$ où $N$ est tué par $\pi$. L'avantage de considérer cet idéal est qu'il vérifie $H'/\mathfrak{m}=H_{B_{0}/A_{0}}$, où l'indice zéro indique la réduction modulo l'idéal maximal.
\medskip
Nous tirons cette définition de \cite[3.2.5]{Tem} :
\begin{defi}\label{conducteur}
Soit $A$ un anneau, $\pi\in A$ qui n'est pas un diviseur de zéro tel que $A$ soit $\pi$-adiquement complet. Soit  $f:X=\Spec(B)\rightarrow Y=\Spec(A)$ un morphisme de présentation finie fidèlement plat. On définit le \textit{conducteur} de $f$ comme étant le plus petit entier $r$ (possiblement infini) tel que $\pi^{r}\in\bh_{B/A}$. En particulier, le conducteur est fini si $f$ est lisse sur le complémentaire de $V(\pi)$.
\end{defi}
$\rmq$ D'après la proposition \ref{gr}, le conducteur ne grandit pas après  changement de base.
En plus, du conducteur de Gabber-Ramero, nous allons également avoir besoin de l'invariant d'Elkik.
\begin{defi}\label{elkinv}
Sous les mêmes hypothèses, on définit l'entier $e^{Elk}_{B/A}$ comme étant le plus petit entier tel que : 
\begin{center}
$\pi^{r}\in H^{Elk}_{B/A}$.
\end{center}
\end{defi}
$\rmq$ On a immédiatement que le conducteur $h$ de Gabber-Ramero est plus petit que l'entier $e^{Elk}_{B/A}$.

Nous supposons de même que pour le lemme \ref{gabinf} que $R$ est artinien avec les mêmes notations, à savoir un idéal $I$ et $\mathfrak{m}$ tels que $I.\mathfrak{m}=0$.
\begin{prop}\label{inftem}
Soient $B$, $B'$ deux $R[[\pi]]$-algèbres finies plates, étales sur $R((\pi))$. 
On pose $\overline{B}:=B/\mathfrak{m}B$. Soit $h$ le conducteur de $\overline{B}/k[[\pi]]$, on considère un entier $n>2h$.
On suppose qu'il existe un isomorphisme $\bar{\nu}:B/\pi^{n}IB\rightarrow B'/\pi^{n}IB'$, alors il existe un isomorphisme $\nu: B\rightarrow B'$ qui est congru à $\bar{\nu}$ modulo $\pi^{n-h}I$.
\end{prop}
\begin{proof}
La preuve est la même que \cite[Prop. 3.3.1]{Tem} en remplaçant \cite[5.4.13]{GR} par le lemme \ref{gabinf} et en utilisant l'idéal $H'$ de la remarque  qui fait suite à la preuve du lemme \ref{gabinf}, au lieu de $H_{B/A}$.
\end{proof}
\subsection{Un résultat de Denef-Loeser-Sebag}
Sans restreindre la généralité, on peut se limiter au cas où $S_{0}=\{x\}$ est réduit à un point et $\la=\la_{x}[x]$.
Soit $\co:=\co_{x}$ de corps résiduel $k$, d'uniformisante $\pi$ et $F$ son corps de fractions.
On note $V^{\la}$ le schéma sur $\Spec(\co)$.
De même, nous avons: 
\begin{center}
$e=e_{x}=\max(e^{Elk}_{V^{\la}/\co},\left\langle 2\rho,\la\right\rangle)$.
\end{center}
Enfin, on note $K:=G(\co)$.
Nous allons avoir besoin à la suite de Denef-Loeser \cite{DL}, de considérer un idéal auxiliaire.
Pour $X$ un schéma réduit pur de dimension $n$, soit $f:Y\rightarrow X$ un morphisme birationnel avec $Y$ lisse, on considère l'idéal Jacobien :
\begin{center}
$\Jac_{f}=\Fitt_{0}(\Omega^{1}_{Y/X})$.
\end{center}
Une autre façon de décrire cet idéal est de considérer la flèche:
\begin{center}
$df:f^{*}\Omega^{n}_{X}\rightarrow\Omega^{n}_{Y}$,
\end{center}
comme $Y$ est lisse, $\Omega^{n}_{Y}$ est localement libre de rang un et l'image de $df$ est de la forme $\Jac_{f}\otimes \Omega^{n}_{Y}$.
Soit $R$ un anneau artinien et $X$ un schéma plat, réduit, de dimension $n$ sur $R[[\pi]]$.
Pour $e\in\mathbb{N}$, on pose: 
\begin{center}
$X^{(e)}(R[[\pi]]):=X(R[[\pi]])\backslash \pi_{e}^{-1}(X_{sing}(\co/\pi^{e+1}\co))$.
\end{center}
où $\pi_{e}$ est la flèche de réduction modulo $\pi^{e+1}$ et $X_{sing}$, le lieu singulier de $X$.
\medskip

Soit $f:Y\rightarrow X$ un morphisme de $R[[\pi]]$-schémas, birationnel sur $R((\pi))$, avec $Y$ lisse, nous considérons le sous-ensemble:
\begin{center}
$\Delta_{e,e'}:=\{y\in Y(R[[\pi]])\vert~ \pi^{e}\in \Jac_{f}(y), h(y)\in X^{(e')}(R[[\pi]])\}$.
\end{center}
Nous avons la proposition suivante tirée de Sebag \cite[Lem. 7.2.1]{Seb}, due à Denef-Loeser \cite{DL} en caractéristique nulle
\begin{prop}\label{loeser}:
Pour tout $n\geq (e,c_{X}e')$, pour tout $z\in\Delta_{e,e'}$ et tout $x\in X(R[[\pi]])$
tel que $h(z)=x~[\pi^{n+1}]$, il existe $y\in Y(R[[\pi]])$ tel que $h(y)=x$ et $z=y~[\pi^{n-e+1}]$.
\end{prop}
Les remarques suivantes sont très importantes pour le reste de la preuve.
\medskip

$\rmqs$\begin{enumerate}
\item
Dans Sebag, la proposition est énoncée pour $k=R$, la preuve s'étend telle quelle une fois que l'on a remplacé la condition $\val(\Jac_{f}(y))=e$ par  $\pi^{e}\in \Jac_{f}(y)$.
\item
En vertu de \cite[4.3.25]{Seb}, la constante $c_{X}e'$ revient précisément à considérer l'idéal $H_{X/R}^{Elk}$ au lieu de $\Jac_{X/R}$. En particulier, dans le cas du schéma $V^{\la}$, on peut remplacer $c_{X}e'$ par $e^{Elk}_{V^{\la}/\co}$.
\end{enumerate}
On déduit de cette proposition le corollaire suivant:
\begin{cor}\label{carreloc}
Soient $\g_{1},\g_{2}\in V^{\la}(R)$ tels que $\g_{1}=\g_{2}~[\pi^{e+1}]$, on considère alors le morphisme de $R[[\pi]]$-schémas:
\begin{center}
$\phi: G\rightarrow V^{\la}$
\end{center}
donnée par $g\mapsto \g_{2}g$, alors il existe $k\in K(R)$ tel que:
\begin{center}
$\g_{1}=k\g_{2}$.
\end{center}
\end{cor}

\begin{proof}
Nous avons $\det(\g_{2})=\pi^{\left\langle 2\rho,\la\right\rangle}\in\Jac_{\phi}$ et $\left\langle 2\rho,\la\right\rangle\leq e$.
On applique alors la proposition \ref{loeser}, avec $f=\phi$, $x=\g_{1}$, $z=1$, $e^{Elk}_{V^{\la}/R[[\pi]]}=c_{X}e'$ d'après la remarque (ii) et $e=\left\langle 2\rho,\la\right\rangle$.
\end{proof}

\section{La flèche global-local}\label{globloc}
Dans cette section, on démontre la proposition \ref{bemolisse}. Sauf mention explicite, on suppose dans ce paragraphe que $G$ est semisimple simplement connexe déployé. On explique dans la section 5.4, les modifications nécessaires  pour le cas général.
Il nous suffit donc, en vertu de \cite[IV. 4. 17.14.2]{EGA}, de vérifier le critère infinitésimal pour un anneau local artinien $R$ d'idéal maximal $\mathfrak{m}$, de corps résiduel $k$ et $I$ un idéal de carré nul de $R$ tel que $I.\mathfrak{m}=0$ et $\bar{R}=R/I$. 

Sans restreindre la généralité, on peut supposer $S_{0}=\{x\}$.
Les données sont alors les suivantes:
\begin{itemize}
\item
Une paire $(\bar{E},\bar{\phi})\in\cmdbD(\bar{R})$,
\item
Une égalité dans $\bar{R}$ entre $(\bar{E}_{x},\bar{\phi}_{x})$ et $(\bar{E}_{0},\bar{\g})$ où l'on peut supposer quitte à localiser que $E_{0}$ est trivial.
\item
Le choix d'une paire $(E_{0},\g_{1})$, $\g_{1}\in V^{\la}_{x}(R)$  qui relève $(\bar{E}_{0},\bar{\g})$.
\end{itemize}
Il nous faut donc relever la paire $(\bar{E},\bar{\phi})$ en $(E,\phi)$ qui s'envoie sur $(E_{0},\g_{1})$ modulo action à gauche par $G_{N}'$ et à droite par $K_{S_{0}}$.
Pour tout $n\in \NN$ et tout point fermé $x$ de $X$, on note $K_{x,n}:=\Ker(G(\co_{x})\rightarrow G(\co_{x}/\pi^{n+1}\co_{x})$.
\subsection{Réduction à un problème semi-local}
Nous avons le  diviseur discriminant $\Delta(\bar{a})\subset X_{\bar{R}}$ de la paire $(\bar{E},\bar{\phi})$.
En regardant la réduction au corps résiduel $\Delta_{0}(\bar{a})$ de $\Delta(\bar{a})$, nous 
avons une décomposition :
\begin{center}
$\Delta_{0}(\bar{a})=\Delta_{0,tr}(\bar{a})+\Delta_{0,sing}(\bar{a})$
\end{center}
et nous notons alors $S':=\supp(\Delta_{0,sing}(\bar{a}))\cup\{x\}$.
On considère l'anneau semi-local complété aux points de $S'$, $\hat{\co}_{X,S'}$ d'anneau total de fractions $F_{S'}$.
On note alors $(\bar{E}_{S'},\bar{\phi}_{S'})$ la paire restreinte au voisinage formel $D_{S'}$, de $S'$ dans $X_{\bar{R}}$, qui s'identifie à $R\hat{\times}\hat{\co}_{X,S'}$.
\begin{lem}\label{locred}
Soit un relèvement local $(E_{S'},\phi_{S'})$ sur $\Spec(\co_{S'})\hat{\times}R$ de $(\bar{E}_{S'},\bar{\phi}_{S'})$ avec $\forall~ s,s'\in S', \chi_{+}(\phi_{s})=\chi_{+}(\phi_{s'})=a\in\abdbD(R)$, alors il existe une paire $(E,\phi)$ qui relève $(\bar{E},\bar{\phi})$ et telle que: 
\begin{itemize}
\item
$a=\chi_{+}(\phi)$.
\item
$(E,\phi)_{\vert D_{S'}}=(E_{S'},\phi_{S'})$.
\end{itemize}
\end{lem}
\begin{proof}
Sur l'ouvert $(X-S')_{\bar{R}}$, la paire restreinte $(\bar{E}_{X-S'},\bar{\phi}_{X-S'})$ est transverse, en particulier, le même argument que la lissité de la flèche:
\begin{center}
$\cm_{\la}^{\diamond}\rightarrow\abdd$,
\end{center}
nous permet de relever la paire $(\bar{E}_{X-S'},\bar{\phi}_{X-S'})$ en une paire $(E_{X-S'},\phi_{X-S'})$ avec $\chi_{+}(\phi)=a'$.
De plus, nous avons un isomorphisme $\bar{\beta}$ entre $(\bar{E}_{S'},\bar{\phi}_{S'})$ et $(\bar{E}_{X-S'},\bar{\phi}_{X-S'})$ sur $\Spec(\bar{R}\hat{\otimes}F_{S'})$, qui revient à la donnée d'un $J_{a}$-torseur sur $\Spec(R\hat{\otimes}F_{S'})$.
Ce $J_{a}$-torseur s'obtient en tirant le $J$-torseur universel $G\times V_{G}^{\la,reg}\rightarrow V_{G}^{\la,reg}\times_{\kcd}V_{G}^{\la,reg}$, où $J$ est le centralisateur régulier.
Comme $J_{a}$ est lisse, par la propriété de relèvement infinitésimal, nous obtenons alors sur $\Spec(R\hat{\otimes}F_{S''})$ un isomorphisme :
\begin{center}
$\beta:(E,\phi)\rightarrow(E_{X-S'},\phi_{X-S'})$.
\end{center}
Maintenant, on recolle  à la Beauville-Laszlo \cite{BL} le triplet  $((E_{X-S'},\phi_{X-S'}), (E_{S'},\phi_{S'}),\beta)$, ce qu'on voulait.
\end{proof}
Il résulte de ce lemme qu'il nous suffit de relever la paire semi-locale $(\bar{E}_{S'},\bar{\phi}_{S'})$.

\subsection{De nécessaires petits lemmes}
Dans ce paragraphe, nous énonçons divers lemmes qui seront nécessaires pour démontrer le relèvement.
\begin{lem}\label{affvide}
Soit $R$ un anneau, soit $I\subset R$ un idéal de carré nul et $\bar{R}=R/I$. Soit $\bar{\g}\in V_{G}^{\heartsuit}(\bar{R}[[\pi]])$ (i.e. génériquement régulier semi-simple), $\g\in V_{G}(R[[\pi]])$ un relèvement de $\bar{\g}$ et $a=\chi_{+}(\g)$. On suppose qu'il existe $\bar{g}\in G(\bar{R}((\pi)))$ 
\begin{center}
$\bar{\g}=\bar{g}^{-1}\eps_{+}(\bar{a})\bar{g}$
\end{center}
alors il existe  $g\in G(R((\pi)))$ un relèvement de $\bar{g}$ tel que:
\begin{center}
$\g=g^{-1}\eps_{+}(a)g$.
\end{center}
\end{lem}
\begin{proof}
Soit $a:=\chi_{+}(\g)$.
On forme le diagramme cartésien suivant :
$$\xymatrix{H_{J}\ar[rr]\ar[d]&& G\times G_{+}^{rs}\ar[d]\\\Spec(R((\pi)))\ar[rr]^-{(\g,\eps_{+}(a))}&&G_{+}^{rs}\times_{\kc^{rs}} G_{+}^{rs}}$$
où la flèche verticale de droite est donnée par $(g,x)\rightarrow (x,gxg^{-1})$, laquelle est un torseur sous le centralisateur régulier,  qui est ici un schéma en tores comme nous sommes au-dessus du lieu régulier semi-simple.
On obtient donc un $J_{a}$-torseur $H_{J}$ sur $\Spec(R((\pi)))$, qui est lisse comme $J_{a}$ est un schéma en tores. Pour obtenir le lemme, il nous faut montrer que ce torseur est trivial.
Or, sur $\bar{R}[[\pi]]$, on a une section qui se relève alors sur $R((\pi))$ par lissité.
\end{proof}

Nous allons voir que ces fibres de Springer ne dépendent que du centralisateur régulier:
\begin{lem}\cite[Lem. 3.5.3]{N}\label{ngocent}
Soit une $k$-algèbre $A$.
Soit $g\in G(A((\pi)))$ et $a\in\kcd(A[[\pi]])^{\heartsuit}$ alors $g^{-1}\eps_{+}(a) g\in V_{G}^{\la}(A[[\pi]])$ si et seulement si $(\pi^{-w_{0}\la},g^{-1}J_{a}(A[[\pi]])g)\in V_{G}^{\la}(A[[\pi]])$.
\end{lem}

Nous terminons le paragraphe par l'introduction des revêtements caméraux qui contrôlent la fibre de Springer, dans la mesure où ils déterminent le centralisateur régulier.
\subsubsection{Les revêtements caméraux}\label{cameralrev}
Soit $R$ local artinien d'idéal maximal $\mathfrak{m}$, $I\subset R$ un idéal de carré nul tel que $I.\mathfrak{m}=0$. 
$\bar{R}:=R/I$ et $k=R/\mathfrak{m}$.
Soient des entiers $n,d$ avec $n>2d$.

Soient $a,a'\in \kcd(R[[\pi]])^{\heartsuit}:=\kcd(R[[\pi]])\cap\mathfrak{C}_{+}^{\la,rs}(R((\pi)))$ tels que:
\begin{itemize}
\item
La réduction sur le corps résiduel du discriminant $\Delta_{0}(a)$ est de valuation inférieure ou égale à $d$.
\item
$a=a'~[\pi^{n}I].$
\end{itemize}

Considérons le  revêtement caméral $X_{a}$  défini par le carré cartésien suivant :
$$\xymatrix{X_{a}\ar[r]\ar[d]&V_{T}^{\la}\ar[d]\\\Spec(R[[\pi]])\ar[r]^-{a}&\kc^{\la}}$$
ainsi que  $X_{a'}$ le revêtement caméral de $a'$.
On rappelle que $V_{T}^{\la}$ et $\kc^{\la}$ sont des schémas tirés sur $\Spec(R[[\pi]])$ par la flèche $-w_{0}\la$.
On note $\bar{a}$ la réduction à $\bar{R}[[\pi]]$ et $X_{\bar{a}}$ la réduction à $\bar{R}[[\pi]]$ du revêtement caméral (et de même pour $a'$).

\begin{lem}\label{gor}
Le revêtement caméral $X_{a}$ est Gorenstein.
\end{lem}
\begin{proof}
Le morphisme $a:X\rightarrow\kcd$ est une immersion régulière, en tant que section d'un fibré vectoriel. Par changement de base plat, la flèche $X_{a}\rightarrow V_{T}^{\la}\rightarrow V_{T}$ est une immersion régulière. Il nous suffit de montrer que $V_{T}$ est Gorenstein.
On peut supposer le groupe semisimple simplement connexe, puisqu'il suffit ensuite ensuite d'ajouter un tore central.
Dans ce cas, le cône associé $C^{*}$ à $V_{T}$ est engendré par les vecteurs $(\alpha_{i},0)$ et $(\omega_{i},\omega_{i})$ d'après \cite[5.2]{Ri} et il résulte de \cite[Lem I.22]{Laf} que l'on peut écrire:
\begin{center}
$C^{*}\cap X^{*}(T)^{+}=\bigcap H_{\sigma_{i}}^{+}$
\end{center}
où $H_{\sigma_{i}}^{+}:=\{x\in X^{*}(T_{+})^{+}\vert~\sigma_{i}(x)\geq 0\}$ et $\sigma_{i}=(\check{\omega}_{i},\check{\omega}_{i})$ où $(\check{\omega}_{i})_{1\leq i\leq r}$ désigne la base duale associée à la base $(\alpha_{i})_{1\leq i\leq r}$.
On obtient alors un morphisme injectif:
\begin{center}
$\sigma:C^{*}\cap X^{*}(T_{+})^{+}\rightarrow\mathbb{N}^{r}$
\end{center}
donné par $x\mapsto\sigma_{i}(x)$.
De plus, l'élément $z=\sum\limits_{i=1}^{r}(\alpha_{i},0)$ vérifie que $\sigma(z)=(1,\dots,1)$ et donc en vertu de \cite[Lem. 2.(iii)]{BR}, la variété $V_{T}$ est Gorenstein.
\end{proof}
Nous commençons avec un cas particulier d'un théorème non publié de Gabber \cite{Gab}:
\begin{thm}\label{gabbi}
On considère la réduction au corps résiduel du revêtement caméral $X_{a}\otimes_{R}k\rightarrow\Spec(k[[\pi]])$. Soit $h$ le conducteur de la flèche et $d$ la valuation du discriminant, alors on a $h\leq d$.
\end{thm}
\begin{proof}
On donne la preuve en appendice.
\end{proof}
\begin{prop}\label{elkik}
Au-dessus de $\Spec(R[[\pi]])$, on a un isomorphisme  $\nu : X_{a'}\rightarrow X_{a}$ qui est congru à l'identité modulo l'idéal $\pi^{n-d}I$.
\end{prop}
\begin{proof}
C'est une application de la proposition \ref{inftem} et du théorème \ref{gabbi}.
\end{proof}

Cela implique le résultat correspondant sur les fibres de Springer:
\begin{prop}\label{affbar}
Il existe $k\in K_{n-d}(R)$ avec $\bar{k}=1$ tel que la multiplication à gauche induit un isomorphisme entre les fibres de Springer affines :
\begin{center}
$k: \kx_{a'}\rightarrow \kx_{a}$.
\end{center}
\end{prop}
\begin{proof}
Il résulte de la description galoisienne que $J_{a}$ et $J_{a'}$ sont entièrement déterminés par les revêtements caméraux.
En particulier, l'isomorphisme $\nu$, donné par le lemme \ref{elkik} entre les revêtements caméraux, induit un isomorphisme :
\begin{center}
$\nu: J_{a'}=I_{\g'_{0}}\rightarrow I_{\g_{0}}=J_{a}$,
\end{center} 
avec dont la réduction modulo $I$ vérifie $\bar{\nu}=Id$. Ici, les éléments $\g_{0}$ et $\g'_{0}$  sont les sections de Steinberg de $a$ et $a'$.
Nous obtenons alors un élément $\nu(\g'_{0})\in I_{\g_{0}}$ tel que $\g_{0}=\nu(\g'_{0})~[\pi^{n-d}I]$, en particulier $\nu(\g'_{0})$ tombe dans l'ouvert $V_{x}^{\la,reg}(R)$ et donc $I_{\g_{0}}=I_{\nu(\g'_{0})}$.
Nous avons alors deux sections:
\begin{center}
$\g_{0},\nu(\g_{0}'):R[[\pi]]\rightarrow V^{\la,reg}$
\end{center}
qui ont même polynôme caractéristique et sont égales modulo $\pi^{n-d}I$.
La flèche $G\times V^{\la,reg}\rightarrow V^{\la,reg}\times_{\kc}V^{\la,reg}$
est un morphisme lisse, il existe alors $k\in K_{n-d}(R)$ avec $\bar{k}=1$ tel que
\begin{center}
$k^{-1}\g_{0}k=\nu(\g'_{0})$
\end{center}
et on en déduit un isomorphisme:
\begin{center}
$\ad(k)^{-1}(J_{a})=J_{a'}$.
\end{center}
On conclut alors par le lemme \ref{ngocent}.
\end{proof}

\subsection{Construction du relèvement}
Nous commençons par montrer que la paire $(\bar{E}_{S'},\bar{\phi}_{S'})$ est isomorphe en fibre générique à sa section de Steinberg.
On note $\bar{a}=\chi_{+}(\bar{\phi})$ et $\bar{\g}_{0}:=\eps_{+}(\bar{a})$, on rappelle que $S'=\supp(\Delta_{0,sing}(\bar{a}))\cup\{x\}$.

Quitte à localiser pour la topologie étale sur $R$, on peut supposer de plus que $\bar{E}_{S'}$ est trivial, la section $\bar{\phi}_{S'}$ correspond alors à la donnée d'une famille: 
\begin{center}
$(\bar{\g}_{s})_{s\in S'}\in V^{\la_{s}}_{s}(\bar{R}[[\pi_{S'}]])$.
\end{center}
Il existe d'après le lemme \ref{affvide} une famille  $(\bar{g}_{s})_{s\in S'}$ telle que :
\begin{equation}
\forall~ s\in S'', \bar{g}_{s}^{-1}\bar{\g}_{0}\bar{g}_{s}=\bar{\g}_{s}.
\label{springfam}
\end{equation}
En particulier, en considérant la fibre de Springer affine :
\begin{center}
$\kx_{\bar{a}, s}:=\{g\in G(\bar{R}((\pi_{s})))/K_{s}(\bar{R})\vert ~ \bar{g}_{s}^{-1}\bar{\g}_{0}\bar{g}_{s}\in V_{s}^{\la_{s}}(\bar{R}[[\pi_{s}]])\}$,
\end{center}
pour $s\in S'$, nous obtenons que la famille  $(\bar{g}_{s})_{s\in S'}$ est dans le produit $\prod\limits_{s\in S'}\kx_{\bar{a}, s}$.
Il est à noter que pour $s\in S'$ avec $s\neq x$, $\bar{\g}_{s}\in K_{s}(\bar{R})$.
On commence par une assertion de surjectivité, fondamentale à notre propos, on rappelle que $S_{0}=\supp(\la)-\{t\}$ où $t$ est le point auxiliaire et que nous pouvons supposer que $S_{0}=\{x\}$.
\begin{prop}\label{borne}
On considère un schéma en groupes $G$ semisimple simplement connexe sur $X$.
Soit $a\in\abdbD(k)$, on a alors $\la\succ2g-2+(e+2d+1)\left|S_{0}\right|+3d_{sing}(a)$.
Considérons le diviseur sur $X$:
\begin{center}
$\Delta_{sing}(a)=\sum\limits_{x\in X}d_{x}[x]$,
\end{center}
et $S':=\supp(\Delta_{sing}(a))\cup\{x\}$, alors  la flèche :
\begin{center}
$\abd\rightarrow\mathfrak{C}_{+,x,e+2d}^{\la}\bigoplus\limits_{s\in S'-\{x\}}\mathfrak{C}_{+,S,2d_{s}}^{\la}$,
\end{center}
est lisse surjective et $\abdbD$ est non vide.
\end{prop}
\begin{proof}
En passant au revêtement étale $X_{\rho}$ de $X$ et comme $G$ est semisimple, $\kcd$ devient isomorphe à
\begin{center}
$\rho^{*}\kcd=\bigoplus\limits_{i=1}^{r}\rho^{*}\mathcal{O}_{X}(\left\langle \omega_{i},-w_{0}\la\right\rangle)$.
\end{center}
Comme $\kcd$ est un facteur direct de $\rho_{*}\rho^{*}\kcd$, il suffit de montrer la surjectivité de la flèche:
\begin{center}
$H^{0}(X_{\rho},\rho^{*}\kcd)\rightarrow \rho^{*}\mathfrak{C}_{+,x,e+2d}^{\la}\bigoplus\limits_{s\in S'-\{x\}}\mathfrak{C}_{+,S,2d_{s}}^{\la}$.
\end{center}
L'assertion de surjectivité résulte alors de Riemann-Roch et de l'inégalité pour tout $i$, 
\begin{center}
$\left\langle\omega_{i}, -w_{0}\la\right\rangle\geq2g-2+(e+2d+1)\left|S_{0}\right|+3d_{sing}(a)$.
\end{center}
Enfin, il suffit de reprendre la preuve de \cite[Lem. 4.7.2]{N} pour obtenir la non-vacuité.
\end{proof}
$\rmq$ Pour la partie locale correspondant à $x$ du polynôme, on aurait pu se contenter de $\mathfrak{C}_{+,x,e+d}$ au lieu de $\mathfrak{C}_{+,x,e+2d}$, mais pour pouvoir être sous les hypothèses de \ref{cameralrev}, il nous fallait un entier supérieur à $2d$.

On peut passer à la preuve de la proposition \ref{bemolisse}:
\begin{proof}
On note $a'_{x}=\chi_{+}(\g_{1})$.
En les points $s\in S'$ avec $s\neq x$, nous avons vu que $\bar{\g}_{s}\in K_{s}(\bar{R})$, en particulier, on peut relever à notre guise cet élément local en un élément 
\begin{center}
$\g'_{s}\in K_{s}(R)$ et $a'_{s}:=\chi_{+}(\g'_{s})$.
\end{center}
Évidemment $a'_{s}$ n'a aucune raison d'être dans $\abdbD(R)$, il nous faut approximer cet élément.
Nous avons le diagramme suivant :
$$\xymatrix{&&\abd\ar[d]\\\Spec(\bar{R})\ar[urr]^{\bar{a}}\ar[r]&\Spec(R)\ar[r]^-{(a'_{s})_{s\in S'}}\ar[r]&\mathfrak{C}_{+,x,e+2d}^{\la}\bigoplus\limits_{s\in S',s\neq x}\mathfrak{C}_{+,s,2d_{s}}^{\la}}.$$
D'après la proposition \ref{borne}, il existe alors un relèvement $a\in\abdbD(R)$ de $\bar{a}$ tel que:
\begin{center}
$a=a'_{x}~[\pi_{x}^{e+2d+1}]$ et $a=a'_{s}~[\pi_{s}^{2d_{s}+1}]$
\end{center}
pour $s\in S'-\{x\}$.
On applique alors le lemme \ref{affbar} pour dire qu'il existe une famille $(k_{x}, (k_{s})_{s\neq x})\in K_{x,e}\times\prod\limits_{S'-\{x\}}K_{s,d_{s}} $ égale à l'identité modulo $\bar{R}$ telle que la multiplication à gauche induit un isomorphisme entre les fibres de Springer affines:
\begin{center}
$\prod\limits_{s\neq x} \kx_{a'_{s},s}\times \kx_{a'_{x},x}\rightarrow \prod\limits_{s\neq x} \kx_{a,s}\times \kx_{a,x}$
\end{center}
La famille $((\g_{s}')_{s\in S'})$ définit d'après le lemme \ref{affvide} une famille  
\begin{center}
$(g_{s})_{s\in S'}\in \prod\limits_{s\in S'}\kx_{a'_{s}, s}$
\end{center}
qui relève la famille  $(\bar{g}_{s})_{s\in S'}$  définie par l'équation \ref{springfam}.
On pose alors pour $s\in S'$
\begin{center}
$\g_{s}=g_{s}^{-1}k_{s}^{-1}\eps_{+}(a)k_{s}g_{s}$.
\end{center}
Pour conclure, nous devons voir que $\g_{x}$ diffère de $\g_{1}$ par un élément de $K_{x}$. Nous avons déjà $\g_{x}=\g_{1}~[\pi^{e+1}]$ et on applique le corollaire  \ref{carreloc}, ce qui conclut.
\end{proof}

\subsection{Le cas réductif et quasi-déployé}\label{casred}
Expliquons maintenant comment on étend le théorème  \ref{transverse} au cas réductif et quasi-déployé.
On commence par conserver l'assertion de semisimplicité, mais on montre l'énoncé pour un groupe quasi-déployé.
Tout d'abord, on a besoin d'une assertion de surjectivité global-local au niveau des polynômes caractéristiques. Cela fait l'objet de la proposition \ref{borne} qui a été prouvée dans le cas quasi-déployé.
Enfin, les énoncés et les preuves s'étendent tels quels à partir du cas déployé, pourvu que l'on remplace $\overline{\Gr}_{\la}$ par $\prod\limits_{\sigma\in\Gamma}\overline{\Gr}_{\sigma\la}$ pour $\la\in\Div^{+}(X,\bt)$ et de même pour le semi-groupe de Vinberg.

En revanche, pour le cas réductif, cela nécessite plus de modifications. 
On traite le cas déployé, le cas quasi-déployé étant analogue.
On rappelle que nous avons une suite exacte:
$$\xymatrix{1\ar[r]&\bg_{der}\ar[r]&\bg\ar[r]^{\det}&\mathbb{G}_{m}^{l}\ar[r]&1}.$$

On considère le champ de Hecke $\cH_{\det(\la)}$ pour le groupe $\mathbb{G}_{m}^{l}$. 
Nous avons alors une flèche:
\begin{center}
$\det:\overline{\cH}_{\la}\rightarrow\cH_{\det(\la)}$
\end{center}
donnée par $(E,E',\beta)\mapsto (\det(E),\det(E'),\det(\beta))$, où l'on pousse les torseurs par la flèche $\det$.
De même, on peut considérer un espace de Hitchin $\cm_{\det(\la)}$ pour $\mathbb{G}_{m}^{l}$ et on obtient alors un diagramme commutatif:
$$\xymatrix{\cmd\ar[d]_{\det}\ar[r]&\overline{\cH}_{\la}\ar[d]^{\det}\\\cm_{\det(\la)}\ar[r]&\cH_{\det(\la)}}$$

On considère alors $\overline{\cH}_{\la}':=\cm_{\det(\la)}\times_{\cH_{\det(\la)}}\overline{\cH}_{\la}$.
Dans un premier temps, nous allons comparer les complexes d'intersection de $\overline{\cH}_{\la}'$ et $\overline{\cH}_{\la}$.
On commence par le lemme suivant:
\begin{lem}\label{lissdet}
La flèche $\cm_{\det(\la)}\rightarrow\cH_{\det\la}$ est lisse.
En particulier, le complexe d'intersection  de $\overline{\cH}_{\la}'$ s'obtient par pullback de celui de $\overline{\cH}_{\la}$.
\end{lem}
\begin{proof}
Le champ $\cm_{\det(\la)}$ classifie les paires $(\cT,\phi)$ constituées d'un $\mathbb{G}_{m}^{l}$-torseur et d'une section $\phi\in\bigoplus\limits_{i=1}^{l} H^{0}(X,\co_{X}(\left\langle \omega_{i}',\la\right\rangle))-\{0\}$. 

Le champ $\cH_{\det(\la)}$ classifie les uplets $(\cT,\cT',(\phi_{i})_{1\leq i\leq l})$ constituées de deux $\mathbb{G}_{m}^{l}$-torseurs et d'isomorphismes
$\phi_{i}:(\omega'_{i})_{*}\cT\rightarrow(\omega'_{i})_{*}\cT'(\left\langle \omega_{i}',\la\right\rangle)$, pour $1\leq i\leq l$, où l'on a poussé les torseurs par $\omega'_{i}$.

En particulier, on obtient que la projection suivant le premier torseur:
\begin{center}
$p_{1}:\cH_{\det(\la)}\rightarrow\Bungm$
\end{center}
est un isomorphisme.
La flèche $\cm_{\det(\la)}\rightarrow\Bungm$ consiste alors  en l'oubli de la section $\phi\in\bigoplus\limits_{i=1}^{l} H^{0}(X,\co_{X}(\left\langle \omega_{i}',\la\right\rangle))-\{0\}$ et est donc lisse.
\end{proof}

Nous avons besoin d'un modèle local. 
On a une flèche:
\begin{center}
$\overline{K\pi^{\la}K}\stackrel{\det}{\rightarrow}\bigoplus\limits_{i=1}^{l}\pi_{S}^{\left\langle \omega_{i}',\la\right\rangle}\co_{S}^{*}$.
\end{center}

On forme alors le carré cartésien:
$$\xymatrix{\overline{K\pi^{\la}K}^{glob}\ar[d]\ar[r]&\overline{K\pi^{\la}K}\ar[d]^{\det}\\ \bigoplus\limits_{i=1}^{l} H^{0}(X,\co_{X}(\left\langle \omega_{i}',\la\right\rangle))-\{0\}\ar[r]&\bigoplus\limits_{i=1}^{l}\pi_{S}^{\left\langle \omega_{i}',\la\right\rangle}\co_{S}^{*}}$$

et on pose alors 
\begin{center}
$\overline{\Gr}_{\la}^{glob}:=\overline{K\pi^{\la}K}^{glob}/K_{der}$,
\end{center}
où $K_{der}=\prod\limits_{x\in X}G_{der}(\co_{x})$.

La flèche $\det$ induit une flèche
\begin{center}
$\overline{\Gr}_{\la}^{glob}\rightarrow\bigoplus\limits_{i=1}^{l} H^{0}(X,\co_{X}(\left\langle \omega_{i}',\la\right\rangle))-\{0\}$,
\end{center}
la fibre est alors isomorphe à $\overline{\Gr}_{\la}$, car  étant donné $f\in
\bigoplus\limits_{i=1}^{l} H^{0}(X,\co_{X}(\left\langle \omega_{i}',\la\right\rangle))-\{0\}$, comme le corps résiduel est algébriquement clos et la caractéristique première à l'ordre de $W$ (l'hypothèse nécessaire sur la caractéristique porte en fait sur $Z_{\bg_{der}}$, mais on vérifie à l'aide des tables que l'ordre de $Z_{\bg_{der}}$ divise l'ordre de $W$), la flèche de projection nous fournit un isomorphisme :
\begin{center}
$\{\g\in\overline{\Gr}_{\la}^{glob}\vert~\det(\g)=f\}\rightarrow \overline{\Gr}_{\la}:=\overline{K\pi^{\la}K/K}$.
\end{center}

La proposition suivante nous permet alors de comparer les complexes d'intersection de $\overline{\cH}_{\la}$ et de $\overline{\Gr}_{\la}^{glob}$.
\begin{prop}
La fibration $\overline{\cH}_{\la}'\rightarrow\Bun$ est localement isomorphe pour la topologie lisse, à $\overline{\Gr}_{\la}^{glob}\times_{k}\Bun$.
\end{prop}
\begin{proof}
La preuve est la même que \cite[A8.c]{Va}.
\end{proof}
Pour obtenir le théorème \ref{transverse}, on construit, comme dans la proposition \ref{bemolisse},  une flèche de $\cmd$ vers  $\overline{\Gr}_{\la}^{glob}$; la preuve de la lissité des flèches est alors analogue, comme on garde la condition globale au niveau des déterminants.

\section{Le lemme de conjugaison}\label{conjlemme} 
Dans cette section, nous démontrons un énoncé de conjugaison, intéressant pour lui-même. Nous le reformulons d'une manière légèrement différente partant d'un élément $x\in K\pi^{\la}K\cap G(F)^{rs}$ au lieu de considérer directement un élément dans $V_{G}^{\la,0}(\co)^{\heartsuit}$.
Dans cette section, on note $F:=k((\pi))$ d'anneau d'entiers $\co$ et de corps résiduel $k$ algébriquement clos de caractéristique première à l'ordre de $W$.

Soit $x\in G(F)$, on rappelle que l'élément $x_{+}=(\pi^{-w_{0}\la},x)$ est dans $V_{G}^{\la}(\co)$ (resp. $V_{G}^{\la,0}(\co)$)  si et seulement si $x\in \overline{K\pi^{\la}K}$ (resp. $K\pi^{\la} K$).
Soit $a_{+}=\chi_{+}(x_{+})$, la valuation $d$ du discriminant $\mathfrak{D}_{\la}(a_{+})$ est donnée par la formule \cite[(7)]{Bt}:
\begin{center}
$d:=\val(\mathfrak{D}_{\la}(a_{+}))=\left\langle 2\rho,\la\right\rangle+d_{0}$
\end{center}
avec 
\begin{center}
$d_{0}=\val(\det(\Id-\ad_{x}:\kg(F)/\kg_{x}(F)\rightarrow\kg(F)/\kg_{x}(F)))$.
\end{center}
Il est à noter que $d_{0}$ peut-être négatif et que $d$ est un entier positif.
Soit $Z$ un schéma  lisse sur $\Spec(\co)$ et $z\in Z(\co)$, on définit: 
\begin{center}
$T_{z}Z(\co):=z^{*}T_{Z/\co}$,
\end{center}
où $T_{Z/\co}$ est le faisceau tangent relatif de $Z/\co$. Le $\co$-module $T_{z}Z(\co)$ est un $\co$-module libre de type fini comme $Z$ est lisse sur $\Spec(\co)$.
De plus, pour un entier $N\in\mathbb{N}$, soit $z_{N}\in Z(\co/\pi^{N}\co)$ l'image de $z\in Z(\co)$, alors nous avons:
\begin{center}
$T_{z_{N}}Z_{N}=T_{z}Z(\co)\otimes_{\co}(\co/\pi^{N}\co)$
\end{center}
où $Z_{N}$ est le $k$-schéma de type fini dont les points pour une $k$-algèbre $A$ sont donnés par:
\begin{center}
$Z_{N}(A)=Z(A\otimes_{k}(\co/\pi^{N}\co))$.
\end{center}
Enfin, si on considère une extension finie $E$ de $F$, d'anneau d'entiers $\co_{E}$, on forme alors le carré cartésien suivant:
$$\xymatrix{Z_{E}\ar[d]\ar[r]^{\pi_{E}}&Z\ar[d]\\\Spec(\co_{E})\ar[r]^{\pi_{E}}&\Spec(\co)}$$
Nous avons alors $T_{Z_{E}/\co_{E}}=\pi_{E}^{*}T_{Z/\co}$ et l'égalité:
\begin{center}
$T_{z}Z_{E}(\co_{E})=T_{z}Z(\co)\otimes_{\co}\co_{E}$.
\end{center}
\begin{thm}[Lemme de conjugaison]\label{conjdi}
Soit $\la\in X_{*}(T)^{+}$.
Soient $x_{+}, y_{+}\in V^{\la,0}(\co)^{\heartsuit}$ tels que $a_{+}=\chi_{+}(x_{+})=\chi_{+}(y_{+})$ et $d=\val(\mathfrak{D}_{\la}(a_{+}))$.
Soit un entier $n>2d$, on suppose que:
\begin{center}
$x_{+}=y_{+}~[\pi^{n}]$,
\end{center}
alors il existe $k\in K_{n-d-1}$ tel que:
\begin{center}
$kx_{+}k^{-1}=y_{+}$.
\end{center}
\end{thm}

\subsection{Espace tangent à l'ouvert lisse $V^{0}$}

Soit l'application $\xi:G\otimes F\rightarrow V^{\la,0}\otimes F$ donnée par $g\mapsto g x_{+}$. Cette application induit un isomorphisme d'espaces vectoriels :
\begin{equation}
d\xi:\kg(F)\rightarrow T_{x_{+}}V^{\la,0}(F).
\label{fibgen}
\end{equation}
Comme $x_{+}$ est un $\co$-point, on a une application $\co$-linéaire:
\begin{center}
$d\xi:\kg(\co)\rightarrow T_{x_{+}}V^{\la,0}(\co)$.
\end{center}
dont la fibre générique est (\ref{fibgen}).
\begin{lem}
On a la formule suivante pour l'indice relatif:
\begin{center}
$[T_{x_{+}}V^{\la,0}(\co):d\xi\kg(\co)]=\left\langle 2\rho,\la\right\rangle$
\end{center}
\end{lem}
\begin{proof}
On considère le réseau:
\begin{center}
$\kq(\co)=(d\xi)^{-1}T_{x_{+}}V^{\la,0}(\co)$
\end{center}
de l'algèbre de Lie $\kg(F)$.
En écrivant $x=(k_{1})^{-1}\pi^{\la}k_{2}$ avec $k_{1},k_{2}\in K$, nous avons $gx\in K\pi^{\la}K$ si et seulement si:
\begin{center}
$g\in\Ad_{k_{1}}(K\pi^{\la}K\pi^{-\la})$.
\end{center}
Le calcul de l'espace tangent en l'identité de $K\pi^{\la}K\pi^{-\la}$ donne alors:
\begin{center}
$\kg(\co)+\ad(\pi^{\la})(\kg(\co))=\kg(\co)+(\kt(\co)\oplus\bigoplus\limits_{\alpha>0}\pi^{\left\langle \alpha,\la\right\rangle}\kg_{\alpha}(\co)\oplus\pi^{-\left\langle \alpha,\la\right\rangle}\kg_{-\alpha}(\co))$.
\end{center}
On en déduit donc l'égalité:
\begin{center}
$\kq(\co)=\ad_{k_{1}}(\kt(\co)\oplus\bigoplus\limits_{\alpha>0}\kg_{\alpha}(\co)\oplus\pi^{-\left\langle \alpha,\la\right\rangle}\kg_{-\alpha}(\co))$.
\end{center}
En particulier, l'indice relatif de $\kq(\co)$ et $\kg(\co)$ est:
\begin{center}
$[\kq(\co):\kg(\co)]=\left\langle 2\rho,\la\right\rangle$.
\end{center}
ce qui conclut.
\end{proof}

On considère l'application:
\begin{equation}
\phi:G\otimes F\rightarrow V^{\la,0}\otimes F
\label{phi}
\end{equation}
donnée par $g\mapsto gx_{+}g^{-1}$ et l'application induite au niveau des espaces tangents $\kg(F)\rightarrow T_{x_{+}}V^{\la,0}(F)$. En composant par $(d\xi)^{-1}:T_{x_{+}}V^{\la,0}(F)\rightarrow\kg(F)$, on obtient un endomorphisme de $\kg(F)$ donné par:
\begin{center}
$\psi=(d\xi)^{-1}\circ d\phi=\Id-\ad_{x}$
\end{center}
où le membre de droite est la différentielle de l'application $g\mapsto gxg^{-1}x^{-1}$. Comme $\Id-\ad_{x}$ est trivial sur $\kg_{x}(F)$ et que $x\in G(F)^{rs}$, il induit un automorphisme sur $\kg^{x}(F):=\kg(F)/\kg_{x}(F)$:
\begin{equation}
\psi^{x}:\kg^{x}(F)\rightarrow\kg^{x}(F)
\label{psigen}
\end{equation}
dont la valuation du déterminant est égale à $d_{0}$.
Comme $x_{+}\in V^{\la,0}(\co)$, on dispose d'une application de $\kg(\co)$ vers $T_{x_{+}}V^{\la,0}(\co)$. En composant par $(d\xi)^{-1}$, on obtient une application $\co$-linéaire:
\begin{center}
$d\xi^{-1}\circ d\phi:\kg(\co)\rightarrow\kq(\co)$.
\end{center}
Elle induit alors une application $\co$-linéaire: 
\begin{center}
$\psi^{x}:\kg^{x}(\co)\rightarrow\kq^{x}(\co)$
\end{center}
avec $\kg^{x}(\co):=\kg(\co)/\kg_{x}(\co)$ et $\kq^{x}(\co):=\kq(\co)/\kq_{x}(\co)$, dont la fibre générique est (\ref{psigen}).
\begin{lem}\label{quotdet}
L'indice relatif de $\kq^{x}(\co)$ et $\psi^{x}(\kg^{x}(\co))$ est majoré par $d$, i.e.:
\begin{center}
$[\kq^{x}(\co):\psi^{x}(\kg^{x}(\co))]\leq d$.
\end{center}
\end{lem}
\begin{proof}
Cela résulte de l'inégalité:
\begin{center}
$[\kq^{x}(\co):\kg^{x}(\co)]\leq[\kq(\co):\kg(\co)]=\left\langle 2\rho,\la\right\rangle$
\end{center}
et de l'égalité $\val(\det(\psi^{x}))=d_{0}$.
\end{proof}
\subsection{Un calcul d'indice pour le centralisateur régulier}

Considérons $\g_{0}:=\eps_{+}(a_{+})$. Il existe alors $g\in G(F)$ tel que $x_{+}=\Ad(g)^{-1}\g_{0}$. On obtient alors un isomorphisme en fibre générique
\begin{center}
$\Ad(g)^{-1}:J_{a_{+}}\otimes F=I_{\g_{0}}\otimes F\rightarrow I_{x_{+}}\otimes F$.
\end{center}
On rappelle que l'on regarde les centralisateurs dans $G$.
Cette flèche induit un isomorphisme au niveau des espaces tangents:
\begin{equation}
\ad(g)^{-1}:\Lie(J_{a_{+}})(F)\rightarrow\kg_{x}(F).
\label{centgen}
\end{equation}
Comme en vertu de la proposition \ref{bouth2}, l'isomorphisme sur $F$ donné par $\Ad(g)^{-1}$, se prolonge en une flèche $J_{a_{+}}\rightarrow I_{x_{+}}\rightarrow G\otimes\co$, on obtient alors une application $\co$-linéaire:
\begin{center}
$\ad(g)^{-1}:\Lie(J_{a_{+}})(\co)\rightarrow\kg_{x}(\co):=\kg(\co)\cap\kg_{x}(F)$,
\end{center}
dont la fibre générique est (\ref{centgen}).
Considérons le morphisme canonique de schémas sur $\Spec(\co)$:
\begin{equation}
\nu:=\chi_{+}\circ\xi\circ\Ad(g)^{-1}: J_{a_{+}}\rightarrow\kcd.
\label{eqlam}
\end{equation}
Considérons le modèle de Néron $J_{a_{+}}^{\flat}$ de $J_{a_{+}}$. Par la propriété universelle du modèle de Néron, on a un morphisme canonique:
\begin{center}
$\iota: J_{a_{+}}\rightarrow J_{a_{+}}^{\flat}$
\end{center}
qui induit l'identité sur $\Spec(F)$.
\begin{lem}
Il existe un morphisme de schémas $\nu^{\flat}:J_{a_{+}}^{\flat}\rightarrow\kcd$ tel que $\nu^{\flat}\circ\iota=\nu$.
\end{lem}
\begin{proof}
Pour $w\in W$ d'ordre $l$, on considère $E=F(\pi^{1/l})$ et $\mathbb{Z}_{l}:=\mathbb{Z}/l\mathbb{Z}$.
On rappelle que la caractéristique est première à l'ordre de $W$.
On note $\tau_{E}$ le générateur de $\Gal(E/F)$ et $\pi_{E}:\Spec(\co_{E})\rightarrow\Spec(\co)$.
Il résulte alors de \cite[sect. 7.1-8.1]{GKM} qu'il existe un élément $w\in W$ d'ordre $l$, tel que le schéma $J_{a_{+}}^{\flat}$ s'identifie à:
\begin{equation}
(\pi_{E,*}T)^{\mathbb{Z}_{l}}
\label{gkmdesc}
\end{equation}
avec $\mathbb{Z}_{l}$ qui agit par $w\tau_{E}$ et l'extension $E/F$ déploie le tore $J_{a_{+}}(F)$.

En particulier il existe $t_{+}\in T^{\la}_{+}(E)$  et $h\in G(E)$ tel que $x_{+}=h^{-1}t_{+}h$.
Comme la flèche $V_{T}^{\la}\rightarrow\kcd$ est finie, il résulte du critère valuatif de propreté que $t_{+}\in V_{T}^{\la}(\co_{E})\cap T_{+}(E)$.
Nous avons alors une flèche 
\begin{equation}
\zeta:T\times_{\co}\co_{E}\rightarrow\kcd\times_{\co}\co_{E}
\label{extzeta}
\end{equation}
donnée par $\g\rightarrow\chi_{+}(\g t_{+})$.
On fait agir $\mathbb{Z}_{l}$ sur $\kcd\times_{\co}\co_{E}$ par $\Id\times\tau_{E}$.
Comme cette flèche entrelace l'action de $w\tau_{E}$ sur $T\times_{\co}\co_{E}$ avec l'action de $\tau_{E}$ sur $\kcd\times_{\co}\co_{E}$, on obtient en prenant les points fixes sous $\mathbb{Z}_{l}$, une flèche:
\begin{center}
$\nu^{\flat}:J_{a_{+}}^{\flat}\rightarrow\kcd$.
\end{center}
Enfin, les flèches $\nu^{\flat}\circ\iota$ et $\nu$ étant les mêmes en fibres génériques, elles sont égales.
\end{proof}

La flèche $\nu^{\flat}:J_{a_{+}}^{\flat}\rightarrow\kcd$ induit une application $\co$-linéaire:
\begin{center}
$d\nu^{\flat}:\Lie J_{a_{+}}^{\flat}(\co)\rightarrow T_{a_{+}}\kcd(\co)$
\end{center}
qui est un isomorphisme en fibre générique, comme $x_{+}$ est génériquement régulier semi-simple.
\begin{lem}\label{flat}
On a la formule suivante pour l'indice relatif:
\begin{center}
$[T_{a_{+}}\kcd(\co):d\nu^{\flat}(\Lie J_{a_{+}}^{\flat}(\co))]=\frac{d+c}{2}$
\end{center}
où $c=\rg T-\rg_{F}J_{a_{+}}(F)$.
\end{lem}
\begin{proof}
On calcule cet indice après extension des scalaires à $\co_{E}$. Nous avons alors l'application linéaire:
\begin{center}
$\Id_{E}\otimes d\nu^{\flat}:\co_{E}\otimes_{\co}\Lie J_{a_{+}}^{\flat}(\co)\rightarrow T_{a_{+}}\kcd(\co_{E})$
\end{center}
qui est la restriction de la différentielle de la flèche (\ref{extzeta}):
\begin{center}
$d\zeta:\kt(\co_{E})\rightarrow T_{a_{+}}\kcd(\co_{E})$
\end{center}
au sous-espace $\co_{E}\otimes_{\co}\Lie J_{a_{+}}^{\flat}(\co)$. 
Il résulte alors de la description (\ref{gkmdesc}) et de \cite[Lem. 3]{B} que:
\begin{center}
$\dim_{k}(\frac{\kt(\co_{E})}{\co_{E}\otimes_{\co}\Lie J_{a}^{\flat}(\co)})=\frac{lc}{2}$.
\end{center}
d'où l'on déduit la formule suivante pour l'indice qui nous intéresse:
\begin{center}
$[T_{a_{+}}\kcd(\co):d\nu^{\flat}(\Lie J_{a}^{\flat}(\co))]=\frac{1}{l}[T_{a_{+}}\kcd(\co_{E}):d\zeta(\kt(\co_{E}))]+\frac{1}{l}\frac{lc}{2}$.
\end{center}
Il ne nous reste donc plus qu'à montrer la formule suivante:
\begin{center}
$[T_{a_{+}}\kcd(\co_{E}):d\zeta(\kt(\co_{E}))]=\frac{ld}{2}$.
\end{center}
On s'est donc ramené au cas déployé et on peut supposer $\co=\co_{E}$ et $t_{+}\in V_{T}^{\la}(\co)\cap T_{+}(F)^{rs}$.
En reprenant les notations de (\ref{extzeta}), on regarde la flèche:
\begin{center}
$d\zeta:\kt(\co)\rightarrow T_{a_{+}}\kcd(\co)$
\end{center}
qui est un isomorphisme en fibre générique. On écrit $t_{+}=(\pi^{-w_{0}\la},t)$.
On identifie alors la base de Steinberg pour $G$, $\kC=\mathbb{A}^{r}$  au fermé $\{1\}\times\mathbb{A}^{r}\subset\kc$.
Pour $\g\in T(\co_{E})$, on a  l'égalité:
\begin{center}
$\chi_{+}(\g t_{+})=\pi^{-w_{0}\la}.\chi(\g t)$,
\end{center}
où $\chi$ est le morphisme de Steinberg pour $G$ et où l'on fait agir le tore central $Z_{+}$ de $G_{+}$ par:
\begin{center}
$z.(1,a_{\bullet})=(\alpha_{\bullet}(z),\pi^{\left\langle \omega_{\bullet},-w_{0}\la\right\rangle}a_{\bullet})$.
\end{center}
En particulier, cette égalité implique la factorisation suivante pour $d\zeta$:
\begin{center}
$d\zeta=d\rho_{-w_{0}\la}\circ d\chi_{t}$,
\end{center}
où l'on a noté $\rho_{-w_{0}\la}$ le morphisme de schémas qui se déduit de l'action de $-w_{0}\la$.
Maintenant, nous avons d'après Steinberg \cite[Lem. 8.2]{S} la formule suivante:
\begin{center}
$[T_{a}(\kC(\co)):d\chi_{t}\kt(\co)]=\frac{d_{0}}{2}$
\end{center}
avec $a=\chi(t)$ et nous avons de plus:
\begin{center}
$[T_{a_{+}}\kcd(\co):\rho_{-w_{0}\la}(T_{a}\kc(\co))]=\left\langle \rho,\la\right\rangle$.
\end{center}
En faisant la somme de ces deux valuations, on obtient que l'indice $[T_{a_{+}}\kcd(\co):d\zeta(\kt(\co))]$ est donné par:
\begin{equation}
\left\langle \rho,\la\right\rangle+\frac{d_{0}}{2}=\frac{d}{2}
\label{diffvin}
\end{equation}
\end{proof}
Nous avons besoin d'une description alternative de $J_{a}^{\flat}$.
\begin{prop}
Soit $\tilde{X}_{a}^{\flat}$ la normalisation de $\tilde{X}_{a}$. Alors le modèle de Néron $J_{a}^{\flat}$ admet la description galoisienne  suivante:
\begin{center}
$J_{a}^{\flat}=\prod\limits_{\tilde{X}_{a}^{\flat}/\bar{X}}(T\times\tilde{X}_{a}^{\flat})^{W}$.
\end{center}
\end{prop}
On écrit $X_{a}=\Spec(\co_{a})$ et $\tilde{X}_{a}=\Spec(\tilde{\co}_{a})$.
\begin{cor}\label{autdim}
On en déduit une autre formule pour la dimension de $\mathcal{P}(J_{a})$:
\begin{center}
$\dime\mathcal{P}(J_{a})=\dime_{\bar{k}}(\mathfrak{t}_{+}\otimes_{\co}\tilde{\co}_{a}^{\flat}/\tilde{\co}_{a})=\frac{d-c}{2}$.
\end{center}
\end{cor}
\begin{proof}
La preuve de la première égalité est la même que \cite[Cor. 3.8.3]{N} et la dernière égalité résulte de \cite[Cor.3.9]{Bt}
\end{proof}
\begin{prop}\label{clecent}
On considère la flèche $d\nu:\Lie J_{a_{+}}(\co)\rightarrow T_{a_{+}}\kcd(\co)$. Alors, on a la formule suivante pour les indices relatifs:
\begin{center}
$[T_{a_{+}}\kcd(\co):d\nu(\Lie J_{a_{+}}(\co))]=d$.
\end{center}
\end{prop}
\begin{proof}
La proposition résulte de la conjonction du diagramme commutatif suivant:
$$\xymatrix{J_{a_{+}}\ar[r]^{\iota}\ar[dr]_{\nu}&J_{a_{+}}^{\flat}\ar[d]^-{\nu^{\flat}}\\&\kcd(\co)}$$
du lemme \ref{flat} ainsi que du fait \ref{autdim}:
\begin{center}
$[\Lie J_{a_{+}}^{\flat}(\co):d\iota(\Lie J_{a_{+}}(\co))]=\frac{d-c}{2}$.
\end{center}
En faisant la somme, le lemme suit.
\end{proof}

\begin{cor}\label{divisib}
Soit $Z\in d\xi(\kq_{x}(\co))\subset T_{x_{+}}V^{\la,0}(\co)$. On suppose 
\begin{center}
$d\chi_{+,x_{+}}(Z)=0~[\pi^{m+d}]$.
\end{center}
Alors, il existe $Z'\in T_{x_{+}}V^{\la,0}(\co)$ tel que $Z=\pi^{m}Z'$.
\end{cor}
\begin{proof}
On rappelle que $\nu=\chi_{+}\circ\xi\circ\Ad(g)^{-1}$ et que l'application 
\begin{center}
$\xi\circ\Ad(g)^{-1}:J_{a_{+}}\rightarrow V^{\la,0}$
\end{center}
est définie sur $\Spec(\co)$.
Il résulte alors de la proposition \ref{clecent} qu'il existe $Z_{1}\in\Lie J_{a_{+}}(\co)$ tel que:
\begin{center}
$d\nu(\pi^{m}Z_{1})=d\chi_{+,x_{+}}(d\xi\circ \ad(g)^{-1}(\pi^{m}Z_{1}))=d\chi_{+,x_{+}}(Z)$.
\end{center}
Or, comme $x_{+}$ est génériquement régulier semisimple, on a un isomorphisme:
\begin{center}
$d\chi_{+,x_{+}}:d\xi(\kg_{x}(F))\rightarrow T_{a_{+}}\kcd(F)$.
\end{center}
Ainsi, en posant $Z_{2}=d\xi\circ \ad(g)^{-1}(Z_{1})\in T_{x_{+}}V^{\la,0}(\co)\cap d\xi(\kg_{x}(F))$, on obtient:
\begin{center}
$Z=\pi^{m}Z_{2}$.
\end{center}
\end{proof}
\subsection{Fin de la preuve}\label{finconjdi}
Nous pouvons maintenant passer à la preuve du théorème \ref{conjdi}:
\begin{proof}
On raisonne par approximations successives, en montrant qu'on peut trouver un élément $k\in K_{n-d-1}$ tel:
\begin{center}
$kx_{+}k^{-1}=y_{+}~[\pi^{n+1}]$.
\end{center}
On pourra donc construire une suite d'éléments de $K_{n-d-1}$ qui converge vers un élément $k_{0}\in K_{n-d-1}$ qui vérifiera:
\begin{center}
$k_{0}x_{+}k_{0}^{-1}=y_{+}$.
\end{center}
Dans l'anneau $\co/\pi^{n+1}\co$, l'idéal $I=(\pi^{n-d})$ est de carré nul car $2(n-d)>n$. Soit $k\in K_{n-d-1}$. 
On a alors une suite exacte:
$$\xymatrix{1\ar[r]&\pi^{n-d}\kg(\co/\pi^{n+1}\co)\ar[r]&G(\co/\pi^{n+1}\co)\ar[r]&G(\co/\pi^{n-d}\co)\ar[r]&1}$$
Ainsi, l'image de $k$ dans  $G(\co/\pi^{n+1}\co)$ définit un élément $\pi^{n-d}X\in\pi^{n-d}\kg(\co/\pi^{n+1}\co)$.
On le relève alors en un élément de $\pi^{n-d}\kg(\co)$, noté de la même manière.
De même $\pi^{n}$ étant de carré nul dans $\co/\pi^{2n}\co$, l'égalité 
\begin{center}
$x_{+}=y_{+}~[\pi^{n}]$
\end{center}
nous fournit l'existence d'un élément $\pi^{n}C\in \pi^{n}T_{x_{+}}V^{\la,0}(\co/\pi^{2n}\co)$.
Comme $\chi_{+}(x_{+})=\chi_{+}(y_{+})$, on obtient:
\begin{equation}
d\chi_{+,x_{+}}(C)=0~[\pi^{n}]
\label{polkar}
\end{equation}
Par lissité de $V^{\la,0}$, on peut relever $\pi^{n}C$ en un élément de $T_{x_{+}}V^{\la,0}(\co)$, noté de la même manière.

En reprenant les notations de (\ref{phi}), on cherche à résoudre l'équation:
\begin{center}
$\phi(k)=y_{+}~[\pi^{n+1}]$,
\end{center}
ce qui se récrit alors comme:
\begin{center}
$\pi^{n-d}d\phi(X)=\pi^{n}C~[\pi^{n+1}]$
\end{center}
soit:
\begin{equation}
d\phi(X)=\pi^{d}C~[\pi^{d+1}]
\label{infconj}
\end{equation}
Maintenant, en appliquant le lemme \ref{quotdet}, il résulte qu'il existe $X\in\kg(\co)$ et $Z\in d\xi(\kq_{x}(\co))$ tel que:
\begin{equation}
d\phi(X)=\pi^{d}C+Z.
\label{presq}
\end{equation}
Pour pouvoir résoudre l'équation (\ref{infconj}), il ne nous reste plus qu'à montrer que:
\begin{center}
$Z=0~[\pi^{d+1}]$.
\end{center}
Il résulte de l'égalité $\chi_{+}\circ\phi=\chi_{+}$ que $d\chi_{+,x_{+}}(d\phi(X))=0$.
Ainsi, en appliquant $d\chi_{+,x_{+}}$ à l'équation (\ref{presq}) et en utilisant l'égalité (\ref{polkar}), on trouve:
\begin{center}
$d\chi_{+,x_{+}}(Z)=0~[\pi^{n+d}]$.
\end{center}
Maintenant il résulte du corollaire \ref{divisib} que $Z=\pi^{n}Z_{2}$, avec $Z_{2}\in T_{x_{+}}V^{\la,0}(\co)$.
Comme $n>d+1$, nous avons une solution à l'équation (\ref{infconj}), ce qui conclut la preuve du théorème \ref{conjdi}.
\end{proof}

\section*{Appendice: Un théorème de Gabber} 
Dans cet appendice, on donne la preuve du théorème \ref{gabbi}, que l'on obtient comme cas particulier de résultats généraux obtenus par Gabber.
Nous suivonst les notes qu'il nous a aimablement communiquées. 
Rappelons l'énoncé du théorème \ref{gabbi}. Soit $k$ algébriquement clos de caractéristique première à l'ordre de $W$.
Soit $\co=k[[\pi]]$, $F$ son corps de fractions et $X=\Spec(\co)$.
On considère $a\in\kcd(\co)\cap\kc^{rs}(F)$ et $X_{a}\rightarrow X$ le revêtement caméral correspondant.
\begin{thm}\label{gabbi2}
Soit $h$ le conducteur de la flèche $X_{a}\rightarrow X$  et $d$ la valuation du discriminant de $a$, alors on a $h\leq d$.
\end{thm}
Nous allons obtenir cet énoncé comme le corollaire d'un résultat plus général sur les anneaux.

Soit $A$ un anneau commutatif, $B$ une $A$-algèbre finie localement libre.
On considère l'extension:
$$\xymatrix{1\ar[r]&J\ar[r]&C:=B\otimes_{A}B\ar[r]^-{m}&B\ar[r]&1}$$
entre $C$-modules.
L'annulateur de cette extension est $m^{-1}(m(\Ann_{C}(J)))$ et donc si nous posons $\cD:=m(\Ann_{C} (J))$, nous avons que $\cD$ annule tous les groupes 
\begin{equation}
\Ext_{C}^{i}(B,-)=0,~ \text{pour}~ i>0.
\label{annil}
\end{equation}
Soit $b\in B$, comme $B$ est un $A$-module de type fini localement libre, on peut considérer alors l'endomorphisme de $A$-modules:
\begin{center}
$\mu_{b}:B\rightarrow B$,
\end{center}
donné par la multiplication par $b$ et on pose $\Tr(b):=\Tr(\mu_{b})$.
Soit le $A$-module $B^{*}:=\Hom_{A}(B,A)$. Il admet une structure de $B$-modules par:
\begin{center}
$ b.f(x)=f(bx)$,
\end{center}
où $f\in B^{*}$ et $b, x\in B$. On considère alors le morphisme de $A$-modules: 
\begin{center}
$\tau_{B/A}:B\rightarrow B^{*}:=\Hom_{A}(B,A)$
\end{center}
donnée par $b\mapsto \Tr(b. )$, où $\Tr(b.)$ est la forme linéaire donnée par $y\mapsto\Tr(by)$.  
On a également un morphisme $A$-linéaire, que l'on peut voir comme une forme bilinéaire:
\begin{center}
$t_{B/A}:B\otimes_{A}B\rightarrow A$.
\end{center}
défini par l'égalité $t_{B/A}(b_{1}\otimes b_{2})=\tau_{B/A}(b_{1})(b_{2})$ pour $b_{1}, b_{2}\in B$.
On considère alors l'idéal \textit{différente}:
\begin{center}
$\delta_{B/A}:=\Ann_{B}(\coker(\tau_{B/A}))$.
\end{center}
L'énoncé est le suivant:
\begin{prop}[Gabber]
On a les inclusions:
\begin{enumerate}
\item
$\cD\subset\delta_{B/A}\subset\Ann_{B}(\Ker(\tau_{B/A}))$.
\item
$\Ann_{B}(\Ker(\tau_{B/A}))\delta_{B/A}\subset\cD$.
\item
L'application $\tau_{B/A}$ est injective si et seulement si $B/A$ est étale sur un ouvert qui contient les idéaux faiblement asssociés (\cite[Déf. 10.63.1]{stack}).
\end{enumerate}
\end{prop}
\begin{proof}
Commençons par montrer (i).
La forme bilinéaire induite par la trace nous donne que  $\Ima (\tau_{B/A})$ est orthogonale à $\Ker(\tau_{B/A})$, d'où l'inclusion $\delta_{B/A}\subset\Ann_{B}(\Ker(\tau_{B/A}))$.
Montrons l'autre inclusion, nous avons un isomorphisme:
$$ \begin{array}{ll}
 h: &B\otimes_{A}B \stackrel{\cong}{\rightarrow} \Hom_{A}(B,B^{*}) \\
   &b_{1}\otimes b_{2} \longmapsto (\phi\mapsto \phi(b_{1})b_{2})
\end{array}$$
qui se restreint en un isomorphisme:
\begin{equation}
h_{1}:\Ann_{C}(J)\rightarrow\Hom_{B}(B^{*},B).
\label{diff}
\end{equation}

Si $x\in\Ann_{C}(J)$, par additivité de la trace d'un endomorphisme d'une suite exacte de modules de type fini localement libres (cf. \cite[Prop. 4.1.4]{GR}), si $B\otimes_{A}B$ est considérée comme une $B$-algèbre par la flèche $b\mapsto1\otimes b$, nous avons:
\begin{center}
$\Tr(x)=m(x)$.
\end{center}
En appliquant cette remarque, nous obtenons que pour $\phi\in\Hom_{A}(B,B^{*})$ qui correspond à un élément $h_{1}^{-1}(\phi)\in\Ann_{C}(J)$, la composition des flèches:
\begin{center}
$B\stackrel{\phi}{\rightarrow}B^{*}\stackrel{\tau_{B/A}}{\rightarrow}B$,
\end{center}
est donnée par la multiplication par $m(h_{1}^{-1}(\phi))$.
Ainsi, nous obtenons l'inclusion 
\begin{center}
$\cD\subset\delta_{B/A}$.
\end{center}
Montrons (ii), si $x.\Ker(\tau_{B/A})=0$ et $y.\coker(\tau_{B/A})=0$ alors la multiplication par $y$ induit un morphisme $B$-linéaire: 
\begin{center}
$B^{*}\rightarrow B/\Ker(\tau_{B/A})$
\end{center}
et la multiplication par $x$ induit une application $B$-linéaire: 
\begin{center}
$B/\Ker(\tau_{B/A})\rightarrow B$.
\end{center}
On obtient alors en vertu de l'isomorphisme (\ref{diff}), un élément de $\Ann_{C}(J)$ et donc: 
\begin{center}
$xy\in\cD$.
\end{center}
Enfin, la dernière assertion résulte de \cite{Fl}.
\end{proof}
On déduit donc de cette proposition que si $B/A$ est étale sur un ouvert qui contient les idéaux faiblement asssociés, alors nous avons:
\begin{equation}
\delta_{B/A}=\cD.
\label{siget}
\end{equation}

\begin{prop}[Gabber]\label{inclgb}
Supposons $B/A$ étale sur un ouvert qui contient les idéaux faiblement asssociés et soit $H_{B/A}$ l'idéal introduit dans \ref{gabber},  alors on a l'inclusion:
\begin{center}
$\delta_{B/A}\subset H_{B/A}$.
\end{center}
\end{prop}
\begin{proof}
On choisit alors une algèbre polynômiale $P$ sur $A$ telle que $B=P/I$ et on considère l'extension de $A$-algèbres:
\begin{equation}
\xymatrix{0\ar[r]&I/I^{2}\ar[r]&P/I^{2}\ar[r]&B\ar[r]&0}
\label{1etoile}
\end{equation}
Soit $x\in\delta_{B/A}:=\Ann_{B}(\coker\tau_{B/A})$, en choisissant un relèvement, il induit donc par multiplication une application :
\begin{center}
$x:I/I^{2}\rightarrow I/I^{2}$.
\end{center}
indépendante du choix du relèvement. On a de plus une application:
\begin{center}
$\partial:I/I^{2}\rightarrow B\otimes_{A}\Omega^{1}_{P/A}$.
\end{center}
Il résulte alors de \cite[5.4.4]{GR} que $x\in H_{B/A}$ si et seulement si l'application
$x:I/I^{2}\rightarrow I/I^{2}$ se factorise par $\partial$.

On pousse alors cette extension par $x$ et on obtient une suite exacte de $A$-algèbres:
\begin{equation}
\xymatrix{0\ar[r]&I/I^{2}\ar[r]&Q\ar[r]&B\ar[r]&0}.
\label{2etoile}
\end{equation}
L'obstruction à relever l'identité $B\rightarrow B$ en un morphisme de $A$-algèbres $B\rightarrow P/I^{2}$ est donnée par l'approche avec la cohomologie de Hochschild par un élément de 
\begin{center}
$\Ext^{2}_{C}(B,I/I^{2})$ 
\end{center}
(cf.\cite[Ch XIV. Thm.2.1]{Ca-E}).
Or, nous avons $\cD=\delta_{B/A}$ et nous avons vu (cf. (\ref{annil})), que $\cD$ tuait tous les groupes $\Ext^{i}_{C}(B,-)$, pour $i>0$.
En particulier, $x$ tue l'obstruction et la suite exacte (\ref{2etoile}) se scinde et nous fournit une dérivation de $Q$ dans $I/I^{2}$.
Maintenant, nous obtenons que la composition
\begin{center}
$I/I^{2}\rightarrow B\otimes_{A}\Omega_{P/A}^{1}\rightarrow B\otimes_{A}\Omega_{Q/A}^{1}\rightarrow I/I^{2}$
\end{center}
est la multiplication par $x$, ce que nous souhaitions.
\end{proof}

Nous pouvons maintenant passer à la preuve du théorème \ref{gabbi2}:
\begin{proof}
Posons $A=\co$ et $B=\Gamma(X_{a},\co_{X_{a}})$.
Nous avons vu que $\delta_{B/A}\subset H_{B/A}$. Maintenant, comme la courbe camérale est Gorenstein d'après \ref{gor}, on sait que l'idéal discriminant est contenu dans l'idéal différente, lui-même dans l'idéal de Gabber-Ramero, d'où l'on déduit $h\leq d$, ce qu'on voulait.
\end{proof}

\begin{flushleft}
Alexis Bouthier \\
Université Paris-Sud UMR 8628\\
Mathématiques, Bâtiment 425, \\
F-91405 Orsay Cedex France \\
E-mail: alexis.bouthier@math.u-psud.fr \\
\end{flushleft}

\end{document}